\documentclass[12pt,psamsfonts,leqno,oneside,letterpaper]{amsart}
\usepackage[dvips,text={6.5truein,9truein},left=1truein,top=1truein]{geometry}
\usepackage{amssymb,amsmath,amscd,enumerate}
\usepackage[pdftex]{graphicx}
\usepackage{url}

\usepackage[colorlinks,linkcolor=blue,citecolor=blue,pdfstartview=FitH]{hyperref}
\input xy
\xyoption{all}
\SelectTips{cm}{12}

\usepackage{color}

\theoremstyle{definition}
\newtheorem{para}{}[section]

\newtheorem{remark}[para]{Remark}
\newtheorem{remarks}[para]{Remarks}
\newtheorem{notation}[para]{Notation}
\newtheorem{convention}[para]{Convention}
\newtheorem{definition}[para]{Definition}
\newtheorem{definitions}[para]{Definitions}

\newcommand\Alternatives{\begin{enumerate}[(i)]}
\newcommand\EndAlternatives{\end{enumerate}}
\newcommand\Conditions{\begin{enumerate}[(1)]}
\newcommand\EndConditions{\end{enumerate}}

\theoremstyle{plain}
\newtheorem{theorem}[para]{Theorem}
\newtheorem{lemma}[para]{Lemma}
\newtheorem{proposition}[para]{Proposition}
\newtheorem{corollary}[para]{Corollary}
\newtheorem{conjecture}[para]{Conjecture}
\newtheorem{claim}[equation]{}

\numberwithin{equation}{para}
\numberwithin{figure}{section}

\newcommand\Number{\begin{para}}
\newcommand\EndNumber{\end{para}}
\newcommand\Definition{\begin{definition}}
\newcommand\EndDefinition{\end{definition}}
\newcommand\Definitions{\begin{definitions}}
\newcommand\EndDefinitions{\end{definitions}}
\newcommand\Theorem{\begin{theorem}}
\newcommand\EndTheorem{\end{theorem}}
\newcommand\Conjecture{\begin{conjecture}}
\newcommand\EndConjecture{\end{conjecture}}
\newcommand\Remark{\begin{remark}}
\newcommand\EndRemark{\end{remark}}
\newcommand\Remarks{\begin{remarks}}
\newcommand\EndRemarks{\end{remarks}}
\newcommand\Convention{\begin{convention}}
\newcommand\EndConvention{\end{convention}}
\newcommand\Notation{\begin{notation}}
\newcommand\EndNotation{\end{notation}}
\newcommand\Lemma{\begin{lemma}}
\newcommand\EndLemma{\end{lemma}}
\newcommand\Proposition{\begin{proposition}}
\newcommand\EndProposition{\end{proposition}}
\newcommand\Corollary{\begin{corollary}}
\newcommand\EndCorollary{\end{corollary}}
\newcommand\Claim{\begin{claim}}
\newcommand\EndClaim{\end{claim}}
\newcommand\Proof{\begin{proof}}
\newcommand\EndProof{\end{proof}}
\newcommand\Equation{\begin{equation}}
\newcommand\EndEquation{\end{equation}}

\newcommand\Bullets{\begin{itemize}}
\newcommand\EndBullets{\end{itemize}}

\newcommand\rhumba{5.1}
\newcommand\easyHeegaard{Lemma 2.1}
\newcommand{\CDSmanganese}{Theorem 6.5}
\newcommand{\CDSgenustwo}{Corollary 5.9}
\newcommand{\CDStrichotomy}{Theorem 5.8}
\newcommand{\CDSnonsep}{Theorem 3.1}
\newcommand{\CDSthreeoeight}{Theorem 6.8}

\newcommand\inter{\mathop{\rm int}}
\newcommand\genus{\mathop{\rm genus}}
\newcommand{\cut}{\,\backslash\backslash\,}
\newcommand\Hg{{\rm Hg}}
\newcommand\kish{\mathop{\rm kish}}
\newcommand\chibar{\bar\chi}
\newcommand\HH{{\mathbb H}}
\newcommand\ZZ{{\mathbb Z}}
\newcommand\myprime[1]{#1'}
\newcommand\tX{\widetilde X}
\newcommand\tj{\widetilde j}
\newcommand\calp{{\mathcal P}}
\newcommand\calb{{\mathcal B}}
\newcommand\calw{{\mathcal W}}
\newcommand\cala{{\mathcal A}}
\newcommand\cald{{\mathcal D}}
\newcommand\calc{{\mathcal C}}
\newcommand\calk{{\mathcal K}}
\newcommand\calh{{\mathcal H}}
\newcommand\caln{\mathcal{N}}
\newcommand\calt{{\mathcal T}}
\newcommand\calq{{\mathcal Q}}

\begin{document}

\title{Volume and topology of bounded and closed hyperbolic $3$-manifolds.}

\author{Jason DeBlois}
\address{Department of Mathematics, Statistics, and Computer Science
(M/C 249)\\
University of Illinois at Chicago\\
851 S. Morgan St.\\
Chicago, IL 60607-7045}
\email{jdeblois@math.uic.edu}
\thanks{Partially supported by NSF grant DMS-0703749}

\author{Peter B.~Shalen}
\address{Department of Mathematics, Statistics, and Computer Science
(M/C 249)\\
University of Illinois at Chicago\\
851 S. Morgan St.\\
Chicago, IL 60607-7045}
\email{shalen@math.uic.edu}
\thanks{Partially supported by NSF grants DMS-0204142 and DMS-0504975}

\begin{abstract}
  Let $N$ be a compact, orientable hyperbolic $3$-manifold with
  $\partial N$ a connected totally geodesic surface of genus $2$.  If
  $N$ has Heegaard genus at least $5$, then its volume is greater than
  $6.89$.  The proof of this result uses the
  following dichotomy: either $N$ has a long \textit{return path}
  (defined by Kojima-Miyamoto), or $N$ has an embedded codimension-$0$
  submanifold $X$ with incompressible boundary $T \sqcup \partial N$,
  where $T$ is the frontier of $X$ in $N$, which is not a book of $I$-bundles.
  As an application of this result, we show that if $M$ is a closed,
  orientable hyperbolic $3$-manifold with $
  \mathrm{dim}_{\mathbb{Z}_2} H_1(M; \mathbb{Z}_2) \geq 5$, and if the
  cup product map $H^1 (M;\mathbb{Z}_2) \otimes H^1(M;\mathbb{Z}_2)
  \rightarrow H^2(M;\mathbb{Z}_2)$ has image of dimension at most one,
  then $M$ has volume greater than $3.44$. 
\end{abstract}

\maketitle

\section{Introduction}

The results of this paper support an old theme in the study of
hyperbolic $3$-manifolds, that the volume of a hyperbolic $3$-manifold
increases with its topological complexity.  The 
first main result, Theorem \ref{vol6.89} below, 
reflects this theme in the context of hyperbolic manifolds with
totally geodesic boundary.  We denote the Heegaard genus of a
$3$-manifold $N$ by $\Hg(N)$.

\newcommand\volsixeightnine{ Let $N$ be a compact, orientable 
hyperbolic
$3$-manifold with $\partial N$ a connected totally geodesic surface of
genus $2$.  
If $\Hg(N) \geq 5$, then $N$ has volume greater than
$6.89$. }
\begin{theorem}\label{vol6.89} \volsixeightnine \end{theorem}

Our second main result concerns closed manifolds:

\newcommand\volthreefourfour{  Let $M$ be a closed, orientable hyperbolic $3$-manifold with
$$  \mathrm{dim}_{\mathbb{Z}_2} H_1(M; \mathbb{Z}_2) \geq 5,  $$
and suppose that the cup product map $H^1 (M;\mathbb{Z}_2) \otimes H^1(M;\mathbb{Z}_2) \rightarrow H^2(M;\mathbb{Z}_2)$ has image of dimension at most one.  Then $M$ has volume greater than $3.44$. }
\begin{theorem}\label{vol3.44} \volthreefourfour
\end{theorem}

Theorem \ref{vol6.89} builds on work by Kojima and Miyamoto.
Miyamoto proved that the minimal--volume compact hyperbolic
$3$-manifolds with totally geodesic boundary of genus $g$ decompose
into $g$ regular truncated tetrahedra, each with dihedral angle
$\pi/3g$ \cite[Theorem 5.4]{Miy}.  Their volumes increase with $g$,
taking values $6.452...$ for $g=2$ and $10.428...$ for $g=3$.
Miyamoto's theorem implies in particular that the minimal volume
compact hyperbolic manifolds with totally geodesic boundary of genus
$g$ have Heegaard genus equal to $g+1$.

Prior to the work of Miyamoto, Kojima-Miyamoto established the
universal lower bound of $6.452...$ for the volume of compact
hyperbolic $3$-manifolds with totally geodesic boundary and described
the minimal--volume examples \cite{KM}.  In fact their result is
slightly stronger: if $N$ is not ``simple'', then $\mathrm{vol}(N) >
6.47$.  In the terminology of \cite{KM}, a hyperbolic
manifold with geodesic boundary is simple if it admits a decomposition
into truncated polyhedra with one internal edge.  Such manifolds are
classified in \cite[Lemma 2.2]{KM}, and include those with minimal
volume.

Theorem \ref{vol6.89} can be regarded as an extension of
Kojima-Miyamoto's theorem.  Experimental results of
Frigerio-Martelli-Petronio \cite{FMP} suggest that the next smallest
manifolds with geodesic boundary after those of minimal volume have
volume greater than $7.1$, so it is likely that Theorem \ref{vol6.89}
is not close to sharp.  Nonetheless it seems to be the only result of
its kind in the literature.

Theorem \ref{vol3.44} will be deduced by combining Theorem
\ref{vol6.89} with results from \cite{CDS} and \cite{CS_vol}. The
transition from these results to Theorem \ref{vol3.44} involves
two other results, Theorems \ref{closedvol6.89} and \ref{genus2or3} below,
which are of independent interest.

\newcommand\closedvolsixeightnine{  Let $M$ be a closed, orientable hyperbolic $3$-manifold containing a closed, connected incompressible surface of genus 2 or 3, and suppose that $\Hg(M) \geq 8$.  Then $M$ has volume greater than $6.89$. }
\begin{theorem} \label{closedvol6.89}  \closedvolsixeightnine
\end{theorem}

Theorem \ref{closedvol6.89} is analogous to \cite[\CDSmanganese]{CDS},
and follows in a similar way: we apply Theorem \ref{vol6.89} and the
results of Miyamoto and Kojima-Miyamoto discussed above, using work of
Agol-Storm-Thurston \cite{ASTD}, to the output of the topological
theorem below.  The notation in the statement is taken from
\cite{CDS}.  In particular, below and in the remainder of this paper,
we will use the term ``simple" as it is defined in \cite[Definitions
1.1]{CDS}, which differs from its usage in \cite{KM} mentioned above.
We also recall the definitions of $M \cut S$ from the first sentence
of \cite{CDS}, and ``$\kish$'' from Definition 1.1 there.  

\newcommand\genustwoorthree
{ Suppose that $M$ is a closed, simple $3$-manifold which contains a
connected closed incompressible surface of genus 2 or 3, and that
$\Hg(M) \geq 8$.  Then $M$ contains a connected closed incompressible
surface $S$ of genus at most 4, such that either $\chibar(\kish(M\cut
S)) \geq 2$, or $S$ is separating and $M \cut S$ has an acylindrical
component $N$ with $\Hg(N) \geq 7$. }
\begin{theorem}\label{genus2or3}  \genustwoorthree
\end{theorem}

Theorem \ref{genus2or3} follows by application of
\cite[\CDStrichotomy]{CDS} jointly with \cite[\CDSnonsep]{CDS}.  It is
the analog of \cite[\CDSgenustwo]{CDS} for manifolds possessing an
incompressible surface of genus 3.

The proof of Theorem \ref{vol3.44} follows the outline of
\cite[\CDSthreeoeight]{CDS}.  In place of the results
concerning 3--free groups used in \cite{CDS}, the proof uses results
of \cite{CS_vol} concerning 4-free groups, and Theorem
\ref{closedvol6.89} above replaces \cite[\CDSmanganese]{CDS}.

All the theorems stated above are proved in Section \ref{sec:closed}.
Sections \ref{sec:prelim}--\ref{sec:111} constitute preparation
for the proof of Theorem
\ref{vol6.89}.  We introduce \textit{return paths}, defined by Kojima
\cite{Ko}, and \textit{$(i,j,k)$ hexagons} in Section
\ref{sec:prelim}.  (The analysis of $(i,j,k)$ hexagons is a crucial
element of \cite{KM}, but we have borrowed our notation for them from
Gabai, Meyerhoff and Milley's paper \cite{GMM}, which is set in a
  different context.)
Lemmas
\ref{borbounds} and \ref{l2vsl1}, which are due to Kojima-Miyamoto
\cite{KM}, respectively give an absolute lower bound on $\ell_1$, and a
lower bound for $\ell_2$ in terms of $\ell_1$.  Lemma \ref{KLM}
refines Lemma \ref{l2vsl1}, giving a bound for $\ell_2$ which improves
that of \cite{KM} when $\ell_1$ is in a certain interval.

Section \ref{sec:KM} describes Kojima-Miyamoto's volume bounds.  The
main result rigorously establishes a lower bound which is apparent
from inspection of \cite[Graph 4.1]{KM}.

\newtheorem*{l1cosh1.215Prop}{Proposition \ref{l1cosh1.215}}
\newcommand\ellonecoshonetwo{ Let $N$ be a hyperbolic $3$-manifold with $\partial N$ connected, totally geodesic, and of genus $2$, satisfying $\cosh \ell_1 \geq 1.215$.  Then $N$ has volume greater than $6.89$.  }
\begin{l1cosh1.215Prop} \ellonecoshonetwo  \end{l1cosh1.215Prop}

\newtheorem*{no(1,1,1)Prop}{Proposition \ref{no(1,1,1)}}
\newcommand\nooneoneone{  Let $N$ be a hyperbolic $3$-manifold with $\partial N$ connected, totally geodesic, and of genus $2$, such that there is no $(1,1,1)$ hexagon in $\widetilde{N}$.  Then $\cosh \ell_1 \geq 1.215$.  }

Also in Section \ref{sec:KM}, Proposition \ref{no(1,1,1)} shows that any manifold with no $(1,1,1)$ hexagon has a shortest return path satisfying $\cosh \ell_1 \geq 1.215$, thus has volume greater than $6.89$ by the above.  In the remaining sections, we explore the topological consequences of the presence of a $(1,1,1)$ hexagon in $\widetilde{N}$, when $\cosh \ell_1 \leq 1.215$.  In Section \ref{sec:111}, we show that under these circumstances, $N$ contains a submanifold $X$ which is a \textit{nondegenerate trimonic manifold relative to $\partial N$} (see Definition \ref{i'm now sir murgatroyd}).

Section \ref{sec:boibs} develops results from the theory of books of $I$--bundles (which were introduced in \cite{ACS}) that we use in Section \ref{sec:trimonic}.  There we introduce trimonic manifolds and prove, in Proposition \ref{it can't happen here}, that such a manifold $X$ does not have the structure of a book of $I$--bundles.  It follows that $X$ has \textit{kishkes} (or \textit{guts}, cf. \cite{ASTD}) with negative Euler characteristic.  Using volume bounds due to Agol-Storm-Thurston, we obtain the following result.

\newtheorem*{hg4orvol7.32Thm}{Theorem \ref{hg4orvol7.32}} 
\newcommand\hgfourorvolseventhreetwo{ Let $N$ be a compact, orientable 
hyperbolic $3$-manifold with $\partial N$ a connected totally geodesic surface of genus 2.  If $\cosh \ell_1 \leq 1.215$ and there is a $(1,1,1)$ hexagon in $\widetilde{N}$, then $\Hg(N) \leq 4$ or $\mathrm{vol}(N) > 7.32$. } 
\begin{hg4orvol7.32Thm}  \hgfourorvolseventhreetwo  \end{hg4orvol7.32Thm}

Together with the results of Section \ref{sec:KM}, this implies Theorem \ref{vol6.89}.


\section{Geometric preliminaries}  \label{sec:prelim}

Suppose $N$ is a hyperbolic $3$-manifold with totally geodesic boundary.  Its universal cover $\widetilde{N}$ may be identified with a convex subset of $\mathbb{H}^3$ bounded by a collection of geodesic hyperplanes.  The following terminology was introduced in \cite{Ko} and used extensively in \cite{KM}; we will use it here as well.

\begin{definition}  Let $N$ be a hyperbolic $3$-manifold with totally geodesic boundary, and let $\widetilde{N} \subset \mathbb{H}^3$ be its universal cover.  A \textit{short cut} in $\widetilde{N}$ is a geodesic arc joining the closest points of two distinct components of $\partial \widetilde{N}$.  A \textit{return path} in $N$ is the projection of a short cut under the universal covering map.  \end{definition}

It is an easy consequence of the definitions that each return path is a homotopically nontrivial geodesic arc properly immersed in $N$, perpendicular to $\partial N$ at each of its endpoints.  Corollary 3.3 of \cite{Ko} asserts that for a fixed hyperbolic manifold $N$ with geodesic boundary and $K \in \mathbb{R}$, there are only finitely many return paths in $N$ with length less than $K$.  Thus the collection of return paths may be enumerated as $\{\lambda_1,\lambda_2,\hdots\}$, where for each $i \in \mathbb{N}$, the length of $\lambda_i$ is less than or equal to the length of $\lambda_{i+1}$.  Fixing such an arrangement, we will denote by $\ell_i$ the length of $\lambda_i$.

It will prove important to understand the distance in $\partial N$, properly interpreted, between endpoints of return paths of $N$.

\begin{definition}  Let $N$ be a compact hyperbolic $3$-manifold with connected totally geodesic boundary, and suppose $\lambda$ is a short cut in $\widetilde{N}$ projecting to $\lambda_i$.  Fix an endpoint $x$ of $\lambda$, and let $\Pi$ be the component of $\partial \widetilde{N}$ containing $x$.  For $j \in \mathbb{N}$ define $d_{ij}$ to be the minimum, taken over all short cuts $\lambda'$ projecting to $\lambda_j$ such that $\lambda'$ has an endpoint $y \in \Pi$ and $\lambda' \neq \lambda$, of $d(x,y)$.  \end{definition}

The requirement above that $\lambda'$ be distinct from $\lambda$ ensures that $d_{ii} > 0$.  In general, $d_{ij}$ is the length of the shortest geodesic arc in $\partial N$ joining an endpoint of $\lambda_i$ to an endpoint of $\lambda_j$.

A crucial tool for understanding the relationships between the lengths $\ell_i$ and distances $d_{ij}$ is a class of totally geodesic hexagons in $\widetilde{N}$ which have short cuts as edges.  The two lemmas below describe the relevant hexagons.

\begin{lemma}  \label{geodhex} Suppose that $\Pi_1$, $\Pi_2$ and $\Pi_3$ are mutually disjoint geodesic planes
in $\HH^3$. For each two-element subset $\{i,j\}$ of $\{1,2,3\}$, let
$\lambda_{ij}$ denote the common perpendicular to $\Pi_i$ and $\Pi_j$. Then
$\lambda_{12}$, $\lambda_{13}$ and $\lambda_{23}$ lie in a common plane $\Pi$.  \end{lemma}

\Proof
We may assume that the three lines $\lambda_{12}$, $\lambda_{13}$ and $\lambda_{23}$ do
 not all coincide, and so by symmetry we may assume that $\lambda_{12}\ne
 \lambda_{13}$. For $i=2,3$ the line $\lambda_{1i}$ meets $\Pi_1$ orthogonally at
 some point $p_i$. Let $L\subset\Pi$ denote the line joining $p_2$ and $p_3$,
 and let $\Sigma$ denote the plane which meets $\Pi$ perpendicularly
 along $L$. It is clear that $\Sigma$ contains $\lambda_{12}$ and $\lambda_{13}$. This
 implies that $\Sigma$ meets the planes $\Pi_2$ and $\Pi_3$
 perpendicularly. For $i=2,3$, let $X_i$ denote the line
 $\Pi_i\cap\Sigma$. Since $\Pi_2\cap\Pi_3=\emptyset$, the lines
 $X_2,X_3\subset\Sigma$ are disjoint. Hence the common perpendicular
 to $X_2$ and $X_3$ is a line $Y\subset\Sigma$. For $i=2,3$, the line
 $Y$ meets the line $X_i\subset\Pi_i$  perpendicularly, and $Y$ is contained
 in the plane $\Sigma$ which is perpendicular to $\Pi_i$; hence $Y$ is
 itself perpendicular to $\Pi_i$. It follows that $Y=\lambda_{23}$. Thus the
 plane $\Sigma$ contains $\lambda_{12}$, $\lambda_{13}$ and $\lambda_{23}$.
\EndProof

\begin{lemma}  \label{rtanghex}  
Let $N$ be a compact hyperbolic $3$-manifold with totally geodesic
boundary, and suppose $\Pi_1$, $\Pi_2$, and $\Pi_3$ are distinct
components of $\partial \widetilde{N}$.  Let $\Pi$ be the plane
containing the short cuts $\lambda_{12}$, $\lambda_{23}$ and
$\lambda_{13}$, which exists by Lemma \ref{geodhex}.  Let $C$ be the
right--angled hexagon in $\Pi$ with edges $\lambda_{ij}$ and the
geodesic arcs in the $\Pi_i$ joining their endpoints.  Then $C \subset
\widetilde{N}$, and $C \cap \partial \widetilde{N} = \cup_i (C \cap
\Pi_i)$.  \end{lemma}

\begin{proof}  $\Pi \cap \widetilde{N}$ is a convex subset of $\Pi$ bounded by the family of disjoint geodesics $\Pi \cap \partial \widetilde{N}$, which includes $\Pi \cap \Pi_i$, $i \in \{1,2,3\}$.  If $\gamma$ is another component of $\Pi \cap \partial \widetilde{N}$, then by definition $\Pi \cap \Pi_1$, $\Pi \cap \Pi_2$, and $\Pi \cap \Pi_3$ are all contained in the component of $\Pi - \gamma$ intersecting $\widetilde{N}$.  Thus there is a single component of $\partial_{\infty} \Pi - \left( \cup_{i =1}^3 \partial_{\infty} \Pi_i \right)$ containing both endpoints of $\gamma$.  Since $\gamma$ is a component of $\Pi \cap \partial \widetilde{N}$, its endpoints are between those of two different geodesics $\Pi \cap \Pi_i$, say $\Pi \cap \Pi_1$ and $\Pi \cap \Pi_2$.  Then since the geodesic containing $\lambda_{12}$ intersects $\Pi_1$ and $\Pi_2$ perpendicularly, it is disjoint from $\gamma$ and contained in the component of $\Pi - \gamma$ intersecting $\widetilde{N}$.  The remainder of $C$ is on the other side of $\lambda_{12}$ from $\gamma$.  Since $\gamma$ was arbitrary, the lemma follows.  \end{proof}

\begin{definition}  Let $N$ be a compact hyperbolic $3$-manifold with totally geodesic boundary, and let $C$ be a right-angled hexagon supplied by Lemma \ref{rtanghex}.  We call the edges of $C$ which are short cuts \textit{internal}, and the remaining edges \textit{external}.  If the internal edges project to $\lambda_i$, $\lambda_j$, and $\lambda_k$, we call $C$ an \textit{$(i,j,k)$ hexagon}.  \end{definition}

This terminology matches that defined in \cite{GMM} in the context of
horospheres and cusped hyperbolic manifolds.  As we will see
below, $(i,j,k)$ hexagons were used extensively in the analysis of
\cite{KM}, although not by name.

The isometry class of a right-angled hexagon is determined by the lengths of three of its pairwise nonadjacent sides.  If $\{\ell, \ell', \ell''\}$ is a collection of such lengths, and $d$ is the length of a side abutting those with lengths $\ell$ and $\ell'$, the \textit{right-angled hexagon rule} (cf. eg. \cite[Theorem 3.5.13]{Ratcliffe}) describes $d$ in terms of the other lengths.  \begin{align}
 \cosh d = \frac{\cosh \ell \cosh \ell' + \cosh \ell''}{\sinh \ell \sinh \ell'}  \label{hexrule} \end{align}  

A prototypical application of the right--angled hexagon rule is the following initial lemma, proved in \cite{KM} during the proof of Lemma 3.2.

\begin{lemma}[Kojima-Miyamoto] \label{d11vsl1} Suppose $N$ is a compact hyperbolic $3$-manifold with connected totally geodesic boundary, and let $R$ be the function of $\ell_1$ defined by the following formula.  \begin{align}  \cosh R = \sqrt{1+ \frac{1}{2\cosh \ell_1 -2}} \label{R} \end{align} 
Then $d_{11} \geq 2R$.  \end{lemma}
  
\begin{proof}  Let $\lambda$ and $\lambda'$ be short cuts in $\widetilde{N}$ with length $\ell_1$ whose feet are at distance $d_{11}$ on some boundary component.  The short cut $\lambda''$ joining the boundary components containing the other feet of $\lambda$ and $\lambda'$ has length $\ell_k$ for some $k \geq 1$.  Applying the right--angled hexagon rule to the $(1,1,k)$ hexagon containing $\lambda, \lambda'$, and $\lambda''$ yields the following inequality. \begin{align}
 \cosh d_{11} = \frac{\cosh^2 \ell_1 + \cosh \ell_k}{\sinh^2 \ell_1} \geq \frac{\cosh^2 \ell_1 + \cosh \ell_1}{\sinh^2 \ell_1} = 1 + \frac{1}{\cosh \ell_1 - 1}  \label{KMd11} \end{align}
The lemma now follows upon applying the ``half--angle formula" for hyperbolic cosine, $\cosh R = \sqrt{(\cosh (2R) + 1)/2}$.  \end{proof}

Note that $R$ is decreasing as a function of $\ell_1$.  Hence using Lemma \ref{d11vsl1}, an upper bound on $d_{11}$ implies a lower bound on $\ell_1$.  An upper bound for $d_{11}$ obtains from area considerations.

\begin{lemma}[\cite{KM}, Corollary 3.5] \label{borbounds}  Let $N$ be a compact hyperbolic $3$-manifold with $\partial N$ connected, totally geodesic, and of genus $2$.  Then the following bounds hold for $d_{11}$ and $\ell_1$.  \begin{align*}
  & \cosh d_{11} \leq 3+2\sqrt{3} & & \cosh \ell_1 \geq \frac{3+\sqrt{3}}{4} \end{align*}  \end{lemma}

\begin{proof}  By Lemma \ref{d11vsl1}, there is a disk of radius $R$ embedded on $\partial N$ around each endpoint of $\lambda_1$, and these two disks do not overlap.  They lift to a radius $R$ disk packing on a component $\Pi$ of $\partial \widetilde{N}$, invariant under the action of $\pi_1 \partial N$.  Bor\"oczky's Theorem \cite{Bor} gives an upper bound $d(R)$ on the local density of a radius $R$ disk packing of $\mathbb{H}^2$.  Since the packing in question is invariant under the action by covering transformations for the compact surface $\partial N$, $d(R)$ bounds the global density of the packing there, yielding the following inequality.
$$ \frac{4\pi(\cosh R -1)}{4\pi} \leq d(R)  $$
The numerator on the left hand side of the inequality above is twice the area of a hyperbolic disk of radius $R$, and the denominator is the area of $\partial N$.  (This follows from the Gauss-Bonnet theorem and the fact that $\partial N$ has genus 2.)  

Let $\alpha$ be the angle at a vertex of a hyperbolic equilateral triangle $T(R)$ with side length $2R$.  $T(R)$ has area $\pi - 3\alpha$, and the intersection with $T(R)$ of disks of radius $R$ centered at its vertices occupies a total area of $3\left( \frac{\alpha}{2\pi} \right) 2\pi(\cosh R - 1)$.  Bor\"oczky's bound $d(R)$ is defined as the ratio of these areas; thus after simplifying the inequality above we obtain the one below.
$$ \cosh R -1 \leq \frac{3\alpha(\cosh R -1)}{\pi - 3\alpha}  $$
Solving for $\alpha$ yields $\alpha \geq \pi/6$.  The hyperbolic law of cosines describes the relationship between $\alpha$ and the side length of $T(R)$:
$$ \cos \alpha = \frac{\cosh^2 (2R) - \cosh (2R)}{\sinh^2 (2R)} = \frac{\cosh (2R)}{\cosh (2R) +1} = \frac{2\cosh^2 R - 1}{2\cosh^2 R}  $$
Using the fact that $\cos \alpha \leq \sqrt{3}/2$ and solving for $\cosh R$ yields $\cosh R \leq (1+\sqrt{3})/\sqrt{2}$.  The inequality for $\cosh d_{11}$ follows using the ``hyperbolic double angle formula", and the inequality for $\cosh \ell_1$ follows upon solving the formula of Lemma \ref{d11vsl1}.
\end{proof}

The following lemma combines Lemmas 4.2 and 4.3 of \cite{KM} with the discussion below them.

\begin{lemma}[Kojima-Miyamoto] \label{l2vsl1}  Let $N$ be a compact hyperbolic $3$-manifold with $\partial N$ totally geodesic, connected, and of genus $2$.  Define quantities $R'$, $E$, and $F$, depending on $\ell_1$, by the following equations.  \begin{align}   & \cosh R' = 3 - \cosh R \label{R'}  \\
   & \cosh E =  \frac{2}{\cosh^2 (R+R') \cdot \tanh^2 \ell_1 - 1} + 1 \label{E} \\
   & \cosh F =  \sqrt{ \frac{\cosh \ell_1 + 1}{\cosh 2R' -1} +1} \label{F}  \end{align}
Then $\ell_2 \geq \max\{\ell_1,\min \{E,F\}\}$.  \end{lemma}

\begin{proof}  The right--angled hexagon rule can be used to obtain lower bounds on $\ell_2$ depending on values for $\ell_1$ and $d_{12}$ or $d_{22}$, respectively. \begin{align}
  &  \cosh \ell_2 \geq  \frac{2}{\cosh^2 d_{12} \tanh^2 \ell_1 - 1} + 1  \label{d12} \\
  &  \cosh \ell_2 \geq \sqrt{ \frac{\cosh \ell_1+1}{\cosh d_{22} - 1} + 1}   \label{d22} \end{align}
This is recorded in Lemma 4.2 of \cite{KM}.  Therefore an upper bound for $d_{12}$ or $d_{22}$ gives a lower bound for $\ell_2$ in terms of $\ell_1$.

Recall from above that an upper bound on $\ell_1$ gives a lower bound of $2R(\ell_1)$, on $d_{11}$, the shortest distance between feet of two different shortest short cuts on (any) component of $\partial \widetilde{N}$.  Hence disks $U$ and $U'$ of radius $R$ in $\partial N$, each centered at a foot of the shortest return path, are embedded and nonoverlapping.  Area considerations imply that  $R'$ is an upper bound for the radii of two equal--size nonoverlapping disks in $\partial N - \mathrm{int}(U \cup U')$ \cite[Lemma 4.3]{KM}.  It follows that at least one of $d_{12} \leq R + R'$ or $d_{22}  \leq 2R'$  holds, since otherwise there would be disks of radius $R'$ embedded around the feet of the second--shortest return path without overlapping $U$ and $U'$.  The inequalities (\ref{d12}) and (\ref{d22}) thus imply that $\ell_2 \geq \min \{E,F\}$.  Note that by definition $\ell_2 \geq \ell_1$, which gives the lemma.\end{proof}

The following lemma contains a new observation improving on the bound of Lemma \ref{l2vsl1} for values of $\ell_1$ with hyperbolic cosine near $1.4$.

\begin{lemma}  \label{KLM}  Let $N$ be a compact hyperbolic $3$-manifold with $\partial N$ totally geodesic, connected, and of genus $2$.  Let $R''$ be determined by the following equation.  \begin{align}
  & \cosh R'' = \frac{1}{\sqrt{2(1-\cos (2\pi/9))}} = 1.4619....  \label{R''} \end{align}
Define quantities $L$ and $M$ depending on $\ell_1$ by \begin{align}
  & \cosh L = \frac{2}{\cosh^2(2R'')\tanh^2 \ell_1 - 1} + 1 \label{L}  \\
  & \cosh M = \sqrt{ \frac{\cosh \ell_1+1}{\cosh(2R'') - 1} + 1}  \label{M}  \end{align}
For any value of $\ell_1$ with \begin{align}
  & \cosh \ell_1 \leq \frac{\cos (2\pi/9)}{2\cos (2\pi/9) - 1} = 1.439..., \label{ell1forR'} \end{align}
$\ell_2$ is bounded below by $\max \{ \ell_1,\min\{E,F\},\min\{L,M\}\}$.  \end{lemma}

\begin{proof}  
Applying Bor\"oczky's theorem as in Lemma \ref{borbounds}, but this
time to four disks of equal radius packed on $\partial N$, we find
that the radius is bounded above by the quantity $R''$
specified by the formula above.
For $\ell_1$ satisfying the bound of (\ref{ell1forR'}), the inequality
(\ref{KMd11}) implies that $d_{11}$ is at least $2R''$.  If each of
$d_{12}$ and $d_{22}$ were also larger than $2R''$, then disks of
radius $R''$ around the feet of both the shortest and second--shortest
return paths would be embedded and nonoverlapping, a contradiction.
Thus $\min\{d_{12},d_{22}\} \leq 2R''$.  Plugging into inequalities
(\ref{d12}) and (\ref{d22}) gives the result.  \end{proof}

The quantities $L$ and $M$ defined above offer a better lower bound on $\ell_2$ than $E$ and $F$, for values of $\ell_1$ with $1.367 \leq \cosh \ell_1 \leq 1.439$.  This is because $E$ and $F$ use the quantity $R'$ of equation (\ref{R'}), and by definition twice the area of a hyperbolic disk of radius $R'$ plus twice the area of a hyperbolic disk of radius $R$ is $4\pi$.  However it is impossible to entirely cover a compact surface of genus 2 with 4 nonoverlapping embedded disks.  Bor\"oczky's theorem bounds the proportion of the area which can be covered by four disks of equal radius, and this supplies $R''$.  This is less than $R'$ for $R$ determined below.
$$ \cosh R \leq 3 - \frac{1}{\sqrt{2(1-\cos (2\pi/9))}}  $$
Solving equation (\ref{R}) for $\cosh \ell_1$, we find that this occurs when $\cosh \ell_1 \geq 1.366...$.  It is at this point that the bound above for $d_{22}$ given by $\cosh(K)$ is better than the bound above given by $2R'$.  The bound on $d_{12}$ given by $\cosh(K)$ becomes better than that given by $R + R'$ somewhat earlier, at least for values of $\ell_1$ with $\cosh \ell_1 \geq 1.25$.  


\section{Volume with a long return path} \label{sec:KM}

By definition, in a hyperbolic manifold $N$ with totally geodesic boundary, $\ell_1/2$ is the height of a maximal embedded collar of $\partial N$.  The volume of such a collar bounds the volume of $N$ below, but leaves a lot out.  The volume bounds of \cite{KM} are obtained by taking a larger collar of $\partial N$ and using separate means to understand the region where it overlaps itself.

\begin{definition}  The \textit{muffin of height $\ell$}, here denoted $\mathit{Muf}_{\ell}$, is the hyperbolic solid obtained by rotating a hyperbolic pentagon with base of length $\ell$, opposite angle $2\pi/3$, and all other angles $\pi/2$, about its base (see \cite[Figure 3.1]{KM}).  \end{definition}

It is a standard fact of hyperbolic trigonometry (see \cite[Theorem
3.5.14]{Ratcliffe}) that positive real numbers $a$, $b$, and $c$
determine a right-angled hexagon in $\mathbb{H}^2$, unique up to
isometry, with alternate sides of lengths $a$, $b$, and $c$.  For
$\ell > 0$, the hexagon specified by $a=b=c=\ell$ 
has an orientation preserving symmetry group of order three which
cyclically permutes the sides of length $\ell$.  The hyperbolic
pentagon mentioned in the definition above is a fundamental domain for
this symmetry group.

For a compact hyperbolic $3$-manifold $N$, a copy of $\mathit{Muf}_{\ell_1}$ is embedded in $\widetilde{N}$ around each short cut with length $\ell_1$.  Lemma 3.2 of \cite{KM} asserts that $\mathit{Muf}_{\ell_1}$ is embedded in $N$ by the universal covering.  Let $A$ be the length of a side joining a vertex with angle $\pi/2$ to the vertex with angle $2\pi/3$ of the pentagon rotated to construct the muffin.  In terms of $\ell_1$, $A$ is given by the formula below.  \begin{align}
  \cosh A = \sqrt{\frac{2}{3}(\cosh \ell_1 + 1)} \label{A}  \end{align}
A collar of $\partial N$ with height less than both $A$ and $\ell_2/2$ has its region of self--overlap entirely contained in $\mathit{Muf}_{\ell_1}$.  This yields the fundamental volume inequality of \cite{KM}, stated there in the proof of Proposition 4.1, which we formulate in the following lemma.

\begin{lemma}[Kojima-Miyamoto]  Let $N$ be a compact hyperbolic $3$-manifold with boundary, and let $H = \min\{A,\ell_2/2\}$.  With $\mathit{Muf}_{\ell_1}$ as defined above and $R$ as in Lemma \ref{d11vsl1}, we have the following bound.  \begin{align} \label{KMvol}
  \mathrm{vol}(N) \geq \mathrm{vol}(\mathit{Muf}_{\ell_1}) + \pi(2 - \cosh R)(2H + \sinh(2H)) \end{align}  \end{lemma}
  
\begin{proof}  The intersection of $\mathit{Muf}_{\ell_1}$ with $\partial N$ is the union of disks $U$ and $U'$ of radius $R$, the quantity defined in Lemma \ref{d11vsl1}.  This is because the pentagon rotated to construct the muffin is a fundamental domain for the orientation--preserving symmetry group of a $(1,1,1)$ hexagon, thus its sides adjacent to the base each have length $R$ (again see \cite[Figure 3.1]{KM}).  It follows from Lemma \ref{d11vsl1} that $U$ and $U'$ are embedded in $\partial N$ without overlapping.

The area of $\partial N - (U \cup U')$ is $4\pi - 4\pi(\cosh R - 1)$.  A collar of $\partial N - (U \cup U')$ of height $H$ is embedded in $N$ without overlapping $\mathit{Muf}_{\ell_1}$, and the bound of the lemma is obtained by adding their volumes.  This uses the following well known formula for the volume $V$ of a collar of height $H$ in $\mathbb{H}^3$ of a set in a plane with area $A$:  $V = A \cdot (2H + \sinh 2H)/4$.  \end{proof}

A formula for the volume of $\mathit{Muf}_{\ell_1}$ is recorded in \cite[Lemma 3.3]{KM}.  The output of inequality (\ref{KMvol}) is recorded as a function of $\ell_1$ in \cite[Graph 4.1]{KM}.  Lower bounds for $H$ are determined at various points on the graph by $A$, the quantities $E$ and $F$ of Lemma \ref{l2vsl1}, and $\ell_1$ itself.  The point on Graph 4.1 of \cite{KM} above $\cosh \ell_1 = 1.215$ is just to the left of the intersection of the curves labeled $H=A$ and $H = F/2$.  A computation gives a volume bound of $7.007\hdots$ here.  To the right of $\cosh \ell_1 = 1.215$, inspection of Graph 4.1 reveals a single local minimum at about $\cosh \ell_1 = 1.4$.  Numerical experimentation indicates that the volume bound at this minimum is just larger than $6.89$.  In this section, our main task is to prove this rigorously.

\begin{l1cosh1.215Prop} \ellonecoshonetwo \end{l1cosh1.215Prop}

\begin{remark}  \label{vol6.94} Although Proposition \ref{l1cosh1.215} probably follows solely from results of \cite{KM}, our proof uses Lemma \ref{KLM}.  In fact, numerical experimentation indicates that the volume bound above could be improved to about $6.94$ with this result.  Since it is easier and suffices for applications, we prove only the bound of $6.89$ here.  \end{remark}

The lemma below proves useful for several results in this section.

\begin{lemma}  For $\frac{3+\sqrt{3}}{4} \leq \cosh \ell_1 \leq 1.4$, the quantity $E$ of (\ref{E}) is decreasing in $\ell_1$.   \label{monotone} \end{lemma}

\begin{proof}  The formula (\ref{R'}) defines $R'$ in terms of $R$ by
$$R' = \cosh^{-1}(3- \cosh R) = \log \left[(3-\cosh R) + \sqrt{(3-\cosh R)^2-1}\right]$$
Taking a derivative with respect to $R$, one finds that $R+R'$ increases with $R$ for $1 < \cosh R < 3/2$, reaching a maximum at $\cosh R = 3/2$, and decreases when $\cosh R > 3/2$.  This implies that $\cosh (R+R')$ is an increasing function of $\ell_1$ on the interval $\frac{3+\sqrt{3}}{4} \leq \cosh \ell_1 \leq 1.4$, since values of $\ell_1$ in this interval give $R$--values between $\frac{3}{2}$ and $\frac{1+\sqrt{3}}{\sqrt{2}}$, and $R$ is a decreasing function of $\ell_1$.  Since the hyperbolic tangent is an increasing function, the lemma follows from the definition of $E$ in Lemma \ref{l2vsl1}.  \end{proof}

In the proof of Proposition \ref{l1cosh1.215}, we divide the interval $1.215 \leq \cosh \ell_1 < \infty$ into subintervals:
$$[1.215, \infty) = [1.215,1.367] \cup [1.367,1.439] \cup [1.439, \infty)$$
We address each subinterval separately.  The first is below.

\begin{lemma}  \label{1.215vol1.367} A compact hyperbolic $3$-manifold $N$ with geodesic boundary satisfying $1.215 \leq \cosh \ell_1 \leq 1.367$ has volume greater than $6.89$.  \end{lemma}

\begin{table} \begin{center} \begin{tabular}{lllll} 
  $\cosh \ell_1$ & muffin volume & $\mathrm{area}(\partial N - (U \cup U'))$ & $H$ & volume  \\ \hline
  $[1.215,1.220]$ & $5.304$ & $2.216$ & $.629$ ($E$) & $6.899$ \\
  $[1.220,1.226]$ & $5.236$ & $2.399$ & $.611$ ($E$) & $6.899$ \\
  $[1.226,1.233]$ & $5.159$ & $2.609$ & $.592$ ($E$) & $6.900$ \\
  $[1.233,1.241]$ & $5.076$ & $2.844$ & $.574$ ($E$) & $6.901$ \\
  $[1.241,1.250]$ & $4.988$ & $3.097$ & $.556$ ($E$) & $6.901$ \\
  $[1.250,1.260]$ & $4.895$ & $3.367$ & $.539$ ($F$) & $6.898$ \\
  $[1.260,1.270]$ & $4.808$ & $3.648$ & $.524$ ($F$) & $6.908$ \\
  $[1.270,1.281]$ & $4.717$ & $3.911$ & $.510$ ($F$) & $6.898$ \\
  $[1.281,1.292]$ & $4.632$ & $4.182$ & $.498$ ($F$) & $6.900$ \\
  $[1.292,1.303]$ & $4.551$ & $4.436$ & $.488$ ($F$) & $6.898$ \\
  $[1.303,1.314]$ & $4.475$ & $4.675$ & $.479$ ($F$) & $6.894$ \\
  $[1.314,1.324]$ & $4.409$ & $4.899$ & $.471$ ($F$) & $6.899$ \\
  $[1.324,1.334]$ & $4.346$ & $5.092$ & $.464$ ($F$) & $6.891$ \\
  $[1.334,1.343]$ & $4.292$ & $5.275$ & $.459$ ($F$) & $6.893$ \\
  $[1.343,1.351]$ & $4.245$ & $5.432$ & $.454$ ($F$) & $6.894$ \\
  $[1.351,1.358]$ & $4.206$ & $5.565$ & $.451$ ($F$) & $6.895$ \\
  $[1.358,1.364]$ & $4.173$ & $5.678$ & $.448$ ($F$) & $6.896$ \\
  $[1.364,1.367]$ & $4.157$ & $5.772$ & $.447$ ($F$) & $6.917$
\end{tabular} \end{center}
\caption{} 
\label{table:volEF}
\end{table}

\begin{proof} The strategy of proof is to break the interval in question into subintervals, bound the constituent quantities in the formula (\ref{KMvol}) on each subinterval, and from this obtain a coarse lower bound for the right hand side of the inequality.  Table \ref{table:volEF} records this computation.  We explain its entries below.

The leftmost column specifies the subinterval of values of $\cosh \ell_1$.  The second column records a lower bound on this interval for the volume of the muffin.  This is attained at the right endpoint, according to \cite[Lemma 3.3]{KM}.  The third column records a lower bound for the area on $\partial N$ of the complement of the base of the muffin.  This is attained at the left endpoint of the subinterval, since $R$ is decreasing in $\ell_1$.  The fourth column records a lower bound for $H$.  This is obtained by computing minima for each of $A$, $E/2$, and $F/2$ on the subinterval and taking the minimum (in each case this lower bound is greater than $\ell_1/2$).  

From its definition (\ref{A}) one easily finds that $A$ is increasing in $\ell_1$, hence a minimum for $A$ is obtained at the left endpoint of each subinterval.  By Lemma \ref{monotone} above, $E$ is decreasing in $\ell_1$ on each subinterval in question, so its minimum is obtained at the right endpoint.  The monotonicity of $F$ is not apparent from its definition, so to find a minimum on each subinterval, we plug the value of $\cosh \ell_1$ at the left endpoint and the value of $\cosh 2R'$ at the right endpoint to the formula (\ref{F}).  In addition to the resulting minimum, we record which of $A$, $E$, and $F$ supplies it in the fourth column of Table \ref{table:volEF}.

The final column assembles these bounds to give a volume bound.  In each column after the first, the decimal approximation has been truncated after three places.   \end{proof}

Using the bounds of Lemma \ref{KLM} for $\ell_2$ in the inequality (\ref{KMvol}) yields the lemma below.

\begin{lemma}  \label{1.367vol1.439}  A compact hyperbolic $3$-manifold with geodesic boundary satisfying  $1.367 \leq \cosh \ell_1 \leq 1.439$ has volume at least $6.89$.  \end{lemma}

\begin{proof}  We assemble a table as in the proof of Lemma \ref{1.215vol1.367}, using $L$ and $M$ to bound $H$ below, instead of $E$ and $F$.  Since $L$ is decreasing in $\ell_1$, its minimum occurs at the right endpoint of each subinterval, whereas the minimum of $M$ occurs at the left endpoint.  The results of the computation are recorded in Table \ref{table:volLM}.  \end{proof}

\begin{table}[ht] \begin{center} \begin{tabular}{lllll}
  $\cosh \ell_1$ & muffin volume & $\mathrm{area}(\partial N - (U \cup U'))$ & $H$ & volume  \\ \hline
  $[1.367,1.377]$ & $4.105$ & $5.818$ & $.447$ ($M$) & $6.892$ \\
  $[1.377,1.392]$ & $4.031$ & $5.966$ & $.448$ ($M$) & $6.894$ \\
  $[1.392,1.416]$ & $3.920$ & $6.176$ & $.449$ ($M$) & $6.893$ \\
  $[1.416,1.439]$ & $3.823$ & $6.485$ & $.451$ ($M$) & $6.959$
\end{tabular} \end{center}
\caption{} 
\label{table:volLM}
\end{table}

We now prove Proposition \ref{l1cosh1.215}.  Below we recall its statement.

\begin{proposition}\label{l1cosh1.215}  \ellonecoshonetwo
\end{proposition}

\begin{proof}  Given Lemmas \ref{1.215vol1.367} and \ref{1.367vol1.439}, we need only concern ourselves with hyperbolic $3$-manifolds $N$ with totally geodesic boundary satisfying $1.439 \leq \cosh \ell_1$.  The union of the muffin and a collar of height $H = \ell_1/2$ has volume given by the following formula.
$$ V(\ell_1) = \pi \left[ \ell_1 + 2\sinh \ell_1 + \sqrt{\frac{2\cosh \ell_1 -1}{2\cosh \ell_1 -2}}\left( \cosh^{-1}\left(\frac{4\cosh \ell_1 +1}{3} \right) - \ell_1 - \sinh \ell_1 \right) \right]  $$
When $\cosh \ell_1 = 1.439$, this yields a volume of $7.1...$  We claim that $V(\ell_1)$ is increasing with $\ell_1$ (as is suggested by Graph 4.1 in \cite{KM}).  Once established, this will complete the proof.

It is clear that $\displaystyle{\sqrt{\frac{2\cosh \ell_1 - 1}{2\cosh \ell_1 - 2}}=\cosh R}$ (recall equation (\ref{R})) is decreasing in $\ell_1$, asymptotically approaching $1$ from above.  The derivative of $\cosh^{-1}((4\cosh \ell_1+1)/3)$ is $1/\cosh R$, bounded above by one.  Hence the quantity in parentheses above,
$$ Q = \cosh^{-1}\left( \frac{4\cosh \ell_1+1}{3} \right) - \ell_1 - \sinh \ell_1,  $$
is decreasing in $\ell_1$.  Its value is negative when $\cosh \ell_1 = 1.439$, and so also on the entire interval in question.  Taking the derivative of $V(\ell_1)$ yields the following.
$$ V'(\ell_1) = \pi \left[ 1 + 2\cosh \ell_1 + \cosh R\left(\frac{1}{\cosh R} - 1 - \cosh \ell_1\right) + \frac{d}{d\ell_1}(\cosh R) \cdot Q \right]  $$
Since $\frac{d}{d\ell_1}(\cosh R)$ and $Q$ are both negative, the above is the sum of a positive number with the product of $1 + \cosh \ell_1$ and $2 - \cosh R$.  When $\cosh \ell_1 = 1.439$, $\cosh R = 1.46... < 2$; since $\cosh R$ is decreasing as a function of $\ell_1$, the derivative of the volume formula is positive for $\cosh \ell_1 \geq 1.439$.
\end{proof}

\begin{lemma}  Let $N$ be a compact hyperbolic $3$-manifold with $\partial N$ connected, totally geodesic, and of genus $2$.  If $\cosh \ell_1 \leq 1.215$, then $\ell_1$ and $A$ are each less than $\ell_2/2$.  \label{l2twicel1} \end{lemma}

\begin{proof} 
Lemma \ref{l2vsl1} implies that $\ell_2 \geq \min \{E,F\}$.  The
values of $\ell_1$ in question here have hyperbolic cosine between
$\frac{3+\sqrt{3}}{4} \cong 1.186$ and $1.215$.  By Lemma
\ref{monotone}, $E$ is monotone decreasing in $\ell_1$ on this
interval; thus a lower bound, obtained by plugging in $1.215$, is
$1.961\ldots$.  The quantity $F$ is not necessarily monotone in $\ell_1$,
since both $\cosh \ell_1$ and $\cosh 2R'$ are increasing in $\ell_1$.
However, a coarse lower bound may be found by substituting
$\cosh^{-1}(1.215)$ for $\ell_1$ in $\cosh \ell_1$ and
$\cosh^{-1}(\frac{3+\sqrt{3}}{4})$ for $\ell_1$ in $\cosh 2R'$.  The
lower bound obtained in this way is $1.960...$.

When $\cosh \ell_1 \leq 1.215$, $\ell_1 \leq .645$.  Taking the
inverse hyperbolic cosine of $1.960$, we find $\ell_2 \geq 1.293$.
The quantity $A$ is clearly increasing in $\ell_1$.  When $\cosh
\ell_1 = 1.215$, the value obtained is $.644....$ This establishes the
lemma.
\end{proof}

\begin{proposition}\label{no(1,1,1)} \nooneoneone  \end{proposition}

\begin{proof}  Suppose $N$ is a hyperbolic $3$-manifold with totally geodesic boundary and no $(1,1,1)$ hexagons.  According to Lemma \ref{l2twicel1}, if $\ell_1$ is at most $1.215$, then $\ell_2$ is at least twice $\ell_1$.  Since $N$ has no $(1,1,1)$ hexagons, in the right--angled hexagon in $\widetilde{N}$ realizing $d_{11}$, the opposite side has length at least $\ell_2$; thus at least twice $\ell_1$.  Using the right--angled hexagon rule as in inequality (\ref{KMd11}), we obtain the following inequality.  \begin{align*}
  \cosh d_{11} & \geq \frac{\cosh^2 \ell_1 + \cosh (2\ell_1)}{\sinh^{2} \ell_1}  = \frac{3\cosh^2 \ell_1 - 1}{\cosh^2 \ell_1 -1} = 3 + \frac{2}{\sinh^2 \ell_1} \end{align*}
We recall from Lemma \ref{borbounds} that $\cosh d_{11} \leq 3 + 2\sqrt{3}$.  Putting this together with the inequality above implies
$$ \sinh^2 \ell_1 \geq \frac{1}{\sqrt{3}}.  $$
This gives $\cosh \ell_1 > 1.255$, contradicting the hypothesis that $\cosh \ell_1 \leq 1.215$.
\end{proof}


\section{Normal books of $I$-bundles}  \label{sec:boibs}

\Number\label{ridin'}

The results of this section and the next are topological. We work in
the PL category in these sections, and follow the conventions of
\cite{CDS}. In particular, we say that a subset $Y$ of a space $X$ is
{\it $\pi_1$-injective} if for all path components $A$ of $X$ and $B$
of $Y$ such that $A\subset B$, the inclusion homomorphism
$\pi_1(A)\to\pi_1(B)$ is injective.  We denote the Euler characteristic of a finite
polyhedron $X$ by $\chi(X)$, and we write $\chibar(X)=-\chi(X)$.

\Proposition\label{easy come}Let $X$ be a compact, orientable
$3$-manifold, and let $S$ and $T$ be components of $\partial X$.
Suppose that for some field $F$ the inclusion homomorphism $H_1(T;F)\to H_1(M;F)$
is surjective. Then $S$ is $\pi_1$-injective in $M$.
\EndProposition

\Proof 
Assume that $S$ is not $\pi_1$-injective. Then there is a
properly embedded disk $D\subset X$ such that $\partial D$ is a
non-trivial simple closed curve in $S$. Let $N$ be a regular
neighborhood of $D$ in $X$, let $A$ denote the component of
$\overline{X-N}$ containing $T$, let $\star$ be a base point in $T$,
and let $G\le\pi_1(X,\star)$ denote the image of $\pi_1(T,\star)$
under the inclusion homomorphism. Then $\pi_1(X,\star)$ is the free
product of $G$ with another subgroup $K$, where $K\cong\ZZ$ if $D$
does not separate $X$, and $K\cong\pi_1(\overline{X-A})$ if $D$ does
separate $X$. In the latter case, $\partial D$ is a separating,
non-trivial simple closed curve in $S$, and hence $\overline{X-A}$ has
a boundary component of strictly positive genus. Thus in either case
we have $H_1(K;F)\ne0$. Hence the inclusion homomorphism
$H_1(A;F)\to H_1(X;F)$ is not surjective. Since $S\subset A$,
this contradicts the hypothesis.
\EndProof

\Lemma\label{how many people cried}
Let 
$\tau:F\to F$ be a
 free involution of a compact orientable surface $F$, and let
 $C\subset F$ be a simple closed curve. Suppose that $\tau(C)$ is isotopic
to a curve which is disjoint from $C$. Then $C$ is isotopic to a curve
$C_1$ such that either (i) $\tau(C_1)\cap C_1=\emptyset$ or (ii)
$\tau(C_1)= C_1$. Furthermore, if $\tau$
reverses orientation we may always choose $C_1$ so that (i) holds.
\EndLemma

\Proof  We fix a metric with convex boundary on
  $F/\langle\tau\rangle$.  Then $F$ inherits a metric such
that $\tau:F \to F$ is an isometry.  Since $C$ is a homotopically
non-trivial simple closed curve in $F$, it is homotopic to a 
curve $C_1$ with shortest length in its homotopy class, 
which is a simple
closed geodesic by \cite[Theorem 2.1]{lilfhs}.  Let
$C'$ be homotopic to $\tau(C)$ and disjoint from $C$.  Then since the
shortest closed geodesics $C_1$ and $\tau(C_1)$ are respectively
homotopic to the disjoint curves $C$ and $C'$, it follows from
  \cite[Corollary 3.4]{lilfhs} that they either are disjoint or
  coincide. (The results we have quoted from \cite{lilfhs} are stated
  there for the case of a closed surface, but it is pointed out in the
  first paragraph of \cite[\S4]{lilfhs} that they hold in the case of
  a compact surface with convex boundary.)

It follows from \cite[Theorem 2.1]{epstein} that
$C$ is isotopic to $C_1$.  This proves
the first assertion.

To prove the second assertion, suppose that $\tau$ reverses
orientation and that $C$ is
isotopic to a curve
$C_1$ such that (ii) holds. Let $A$ be an invariant annular
neighborhood of $C_1$. Since $\tau$ is a free involution it must
preserve an orientation of the invariant curve $C$; since it reverses
an orientation of $F$, it must therefore interchange the components of
$\partial A$. Hence (i) holds if $C_1$ is replaced by one of the
components of $\partial A$.
\EndProof

If a $3$-manifold $X$ has the structure of an $I$-bundle over a surface $T$ 
and $p: X \rightarrow T$ is the bundle projection, we will call 
$\partial_v X \doteq p^{-1}(\partial T)$ the \textit{vertical} boundary
of $X$ and $\partial_h X \doteq \overline{\partial X - \partial_v X}$ 
the \textit{horizontal} boundary of $X$.  Note that $\partial_v X$ inherits
the structure of an $I$-bundle over $\partial T$, and $\partial_h X$
 the structure of a $\partial I$-bundle over $T$, from the original $I$-bundle 
structure on $X$.  We call
an annulus $A \subset X$ \textit{vertical} if $A = p^{-1}(p(A))$.

Let $M$
 an orientable, irreducible
$3$-manifold $M$, and let $F$ be a $\pi_1$-injective $2$-dimensional
submanifold of $\partial M $.
By an {\it essential annulus or torus in $M$ relative to $F$}
we mean a properly embedded annulus or torus in $M$
which is $\pi_1$-injective, has its boundary contained in $F$, and not
parallel to a subsurface of $F$. 

In the case where $M$ is boundary-irreducible, we shall
define an {\it essential annulus or torus in $M$} to be
an  essential annulus or torus in $M$ relative to $\partial M$.

\EndNumber

\Proposition\label{one bundle at a time}
Let  $X$ be an $I$-bundle over a compact surface, and suppose that $A$ 
is an essential annulus in $|\calw|$ relative to $\partial_hX$. Then $A$ is 
isotopic, by an ambient isotopy of $X$ which is constant on
$\partial_v X$,  to a
vertical annulus in $X$.
\EndProposition

\Proof 
Let $j:A\to X$ denote the inclusion. If $j$ is homotopic to a map of
$A$ into $\partial_hX$ then by
\cite[Lemma 5.3]{Waldhausen}, $A'$ is boundary parallel, contradicting
essentiality relative to $\partial_hX$. Hence $j$ is not homotopic to a
map of $A$ into $\partial_hX$.

It follows that if  $X$ is a trivial $I$-bundle then $A$ has one boundary
component on each component of $\partial_h X$. Lemma 3.4 of
\cite{Waldhausen} then asserts that $A$ is ambiently isotopic in $X$ to a
vertical annulus, by an isotopy fixing $\partial_v X$ and one
component of $\partial_h X$. This gives the conclusion in the case
where $X$ is a trivial $I$-bundle.

We turn to the case where $X$ is a twisted $I$-bundle. Then for some
compact orientable surface $F$ and some orientation-reversing free
involution $\tau:F\to F$ we may write $X=(F\times I)/\hat\tau$, where
$\hat\tau:F\times I\to F\times I$ is defined by
$\hat\tau(x,t)=(\tau(x),1-t)$. The quotient map $p:F\times I\to X$ is
a two-sheeted covering map, and $p$ maps $F\times\{i\}$
homeomorphically onto $\partial_h X$ for $i=0,1$. Hence $j:A\to X$ may
be lifted to an embedding $\tj:A\to\tX$. Since $j$ is not homotopic to
a map of $A$ into $\partial_hX$, the annulus $\tj(A)$ is essential. By
the case of the proposition already proved, $\tj$ is ambiently
isotopic in $\tX$ to an embedding $\tj'$ such that $\tj'(A)$ is
vertical.

For $i=0,1$ let  $C_i$  denote the component of $\partial
A$ with $\tj'(C_i)\subset F\times\{i\}$. Define a simple closed curve
$C\subset F$ by $C\times\{0\}=\tj'(C_0)$. Then $C\times\{0\}$ is isotopic
to $\tj(C_0)$. Since $\tj'(A)$ is
vertical, 
$\tj'(C_1) = C \times \{1\}$; hence by definition we have
$\tau(C)\times\{0\}=
\hat\tau(\tj'(C_1))$, and so $\tau(C) \times \{0\}$ is isotopic
to $\hat\tau(\tj(C_1))$. But since $j$ is an embedding, $\tj(C_0)$ and
$\hat\tau(\tj(C_1)$ are disjoint curves in $F\times\{0\}$.  Thus $C$
and $\tau(C)$ are isotopic in $F$ to disjoint curves, and
it follows from Lemma \ref{how many people cried} that $C$ is isotopic
in $F$ to a curve $C_1$ such that $\tau(C_1)\cap
C_1=\emptyset$. This implies that $\tj'$ is ambiently isotopic to an
embedding $\tj'_1$ such that $\tj'_1(A)$ is a vertical annulus and
$\tau(\tj'_1(A))\cap \tj'_1(A)=\emptyset$.

Hence $p\circ\tj'_1:A\to X$ is an embedding, and $A_1\doteq 
p\circ\tj'_1(A)$ is a vertical annulus in $X$. Since 
$\tj'_1$ is ambiently isotopic in $F \times I$ to 
$\tj$, the homeomorphism $p\circ\tj'_1:A\to A_1$ is homotopic to $j$ by
a boundary-preserving homotopy in $X$. It then follows from
\cite[Corollary 5.5]{Waldhausen} that $A$ is isotopic to $A_1$
by an ambient isotopy fixing $\partial_v X$.
\EndProof

\Number\label{oshkosh bighash}
Let $W$ be a compact, orientable, irreducible and boundary-irreducible
$3$-manifold with $\partial W\ne\emptyset$. We recall the definition
of the characteristic submanifold $\Sigma$ of $W$ relative to
$\partial W$. Up to ambient isotopy, $\Sigma$ is the unique compact
submanifold of $W$ with the following properties.
\begin{enumerate}
\item\label{thermometer} Every component of $\Sigma$ is either an $I$-bundle $P$ over a
surface such that $P\cap\partial W=\partial_hP$, or a Seifert
fibered space $S$ such that $S\cap\partial W$ is a saturated
$2$-manifold in $\partial S$.
\item\label{sushi} Every component of the frontier of $\Sigma$ is an
essential annulus or torus in $W$.
\item\label{cheesequake} No component of $\Sigma$ is ambiently isotopic in $W$ to a
submanifold of another component of $\Sigma$.
\item\label{big ol' hash} If $\Sigma_1$ is a compact
submanifold of $W$ such that (\ref{thermometer}) and (\ref{sushi}) hold with $\Sigma_1$
in place of $\Sigma$, then $\Sigma_1$ is ambiently isotopic in $W$ to
a submanifold of $\Sigma$.
\end{enumerate}
For further details, see \cite{Jo} and \cite{JaS}.

It follows from Property (\ref{sushi}) of $W$ that $\Sigma$ is
$\pi_1$-injective. Hence if $W$ is simple, the fundamental group of a
component of $\Sigma$ cannot have a rank-$2$ free abelian subgroup. It
follows that if $W$ is simple then every Seifert-fibered component of
$\Sigma$ is a solid torus. In particular, all the components of the
frontier of $\Sigma$ are essential annuli in this case.

\EndNumber

The rest of this section is concerned with books of $I$-bundles. We
shall follow the conventions of \cite{CDS}, and we refer the reader to
\cite[\rhumba]{CDS} for the definition of a book of $I$-bundles. As in
\cite{CDS} we shall denote the union of the pages of a book of
$I$-bundles by $\calp_\calw$ and the union of its bindings by
$\calb_\calw$, and we shall set $|\calw|=\calp_\calw\cup\calb_\calw$
and $\cala_\calw=\calp_\calw\cap\calb_\calw$.

\Definition Let $\calw$ be a book of $I$-bundles, and set
$W=|\calw|$, and let $C$ denote a regular neighborhood of
$\cala_\calw$ in $W$.  We shall say that $\calw$ is {\it normal} if
(i) $W$ is a simple $3$-manifold, and (ii) $\overline{W-C}$ is
ambiently isotopic in $W$ to the characteristic submanifold of the
pair $(W,\partial W)$.
\EndDefinition

\Proposition\label{all god's chillun got guts} Let $W$ be a simple
$3$-manifold.  Let $\Sigma$ denote the characteristic submnanifold of
the pair $(W,\partial W)$, and suppose that
$\chi(\overline{W-\Sigma})\ge0$. Then $\overline{W-\Sigma}$ is a
regular neighborhood of a properly embedded submanifold $\cala$ of
$W$, each component of which is an annulus. Furthermore, there is a
normal book of $I$-bundles $\calw$ such that $|\calw|=W$ and
$\cala_\calw=\cala$.
\EndProposition

\Proof
Set $\calc=\overline{W-\Sigma}$.
As we observed in \ref{oshkosh bighash}, the components of the
frontier of $\Sigma$ are $\pi_1$-injective annuli. In particular, no
component of $\partial\calc$ is a $2$-sphere. Since
$\chi(\calc)\ge0$, it follows that every component of $\partial\calc$
is a torus. On the other hand, the property (\ref{sushi}) of
$\Sigma$ stated in  \ref{oshkosh bighash} implies that $\calc$ is
$\pi_1$-injective in $W$. Since $W$ is simple, it follows that the
fundamental group of a component of $\calc$ cannot have a rank-$2$
free abelian subgroup. Hence the components of $\calc$ are
boundary-reducible, and in view of the irreducibility of $W$ they must
be solid tori. 

Let $C$ be any component of $\calc$. The frontier components of $C$
are among the frontier components of $\Sigma$, and we observed in
\ref{oshkosh bighash} that these are essential annuli in $W$. In particular
it follows that the components of $C\cap\partial W$ are
$\pi_1$-injective annuli on the solid torus $C$. Hence $C$ may be
given the structure of a Seifert fibered space in such a way that
$C\cap\partial W$ is saturated. It now follows from Property (\ref{big
ol' hash}) of $\Sigma$ that $C$ is ambiently isotopic in $W$ to a
submanifold of $\Sigma$. Since $C$ is also ambiently isotopic to a
submanifold of $W-\Sigma$, it must be ambiently isotopic to the
regular neighborhood of a frontier component of $\Sigma$.
This proves:
\Claim\label{spitzer}
$\calc$ is a regular neighborhood of a properly embedded
$2$-manifold  $\cala\subset W$ whose components are annuli. 
In particular, each component of
$\calc$ has two frontier annuli. 
\EndClaim

In particular this gives the first assertion of the proposition.

Let $\Sigma_0$ denote the union of all components of $\Sigma$ that are
solid tori, and set $\Sigma_-=\Sigma-\Sigma_0$. Since every component
of $\Sigma$ is a solid torus or an $I$-bundle $P$ with $P\cap\partial
W=\partial_hP$, each component of $\Sigma_-$ is an $I$-bundle $P$ over a
surface of negative Euler characteristic with $P\cap\partial W=\partial_hP$.

We now claim:
\Claim\label{morally triangulated}
Each component of
$\calc$ has one frontier annulus contained in $\Sigma_0$ and one
contained in $\Sigma_-$.
\EndClaim

Let $C$ be any component  of $\calc$. According to \ref{spitzer}, $C$
has two frontier annuli $A_1$ and $A_2$. Let $Q_i$ denote the
component of $\Sigma$ containing $A_i$ (where a priori we might have
$Q_1=Q_2$). To prove \ref{morally triangulated} we must show that
$Q_1$ and $Q_2$ cannot both be contained in
$\Sigma_-$ or both be contained in $\Sigma_0$.

First suppose that the $Q_i$ are both contained in $\Sigma_-$. Then
$Q\doteq Q_1\cup C\cup Q_2$ may be given the structure of an $I$-bundle over a
surface in such a way that $\partial_hQ=Q\cap\partial W$. It therefore
follows from the property (\ref{big ol' hash}) of $\Sigma$ stated in
\ref{oshkosh bighash} that $Q$ is ambiently isotopic to a submanifold
of $\Sigma$. But since
$\chi(Q_i)<0$ for $i=1,2$ we have $\chibar(Q)>\chibar(Q_i)$ for
$i=1,2$. Hence $Q$ cannot be ambiently isotopic to a submanifold
of $Q_i$ for $i=1,2$. Furthermore, if $Q'$ is a component of $\Sigma$
distinct from $Q_1$ and $Q_2$, then since $Q\cap Q'=\emptyset$ and
$\chi(Q)<0$, the $I$-bundle $Q$ cannot be ambiently isotopic to a submanifold
of $Q'$. This is a contradiction.

Now suppose that the $Q_i$ are both contained in $\Sigma_0$. Then the
$Q_i$ are solid tori, and the image of the inclusion homomorphism
$\pi_1(A_i)\to\pi_1(Q_i)$ has some finite index $m_i$ in $\pi_1(Q_i)$
for $i=1,2$. The fundamental group of
$Q\doteq Q_1\cup C\cup Q_2$ has presentation $\langle x_1,x_2:
x_1^{m_1}=x_2^{m_2}\rangle$. But since $W$ is simple and the frontier
annuli of $Q$ are essential, $\pi_1(Q)$ has no rank-$2$ free abelian
subgroup. Hence at least one of the $m_i$ must be equal to $1$, and we
may assume that $m_2=1$. But this implies that $Q_2$ is ambiently
isotopic to a submanifold of $Q_1$, which contradicts the property
(\ref{cheesequake}) of $\Sigma$ stated in \ref{oshkosh bighash}. This
completes the proof of \ref{morally triangulated}.

It follows from \ref{spitzer} and \ref{morally triangulated} that
$\calb\doteq\Sigma_0\cup\calc$ is a regular neighborhood of $\Sigma_0$
in $W$, and that the frontier $\cala'$ of $\calb$ is ambiently
isotopic to $\cala$. If we set $\calp=\Sigma_-$, it follows from the
definition given in \cite[\rhumba]{CDS} that $\calw'=(W,\calb,\calp)$
is a book of $I$-bundles. Normality is immediate from the
construction. Since $\cala_{\calw'}=\cala'$ is ambiently isotopic to
$\cala$, there is a normal book of $I$-bundles $\calw$ with
$|\calw|=|\calw'|=W$ and $\cala_\calw=\cala$.
\EndProof

Requiring that a book of $I$-bundles structure   
be normal rules out certain degeneracies.
For example, if $W^\flat$ is an $I$-bundle over a closed surface of negative
Euler characteristic, we may write $W^\flat=|\calw^\flat|$ for some book of
$I$-bundles $\calw^\flat$ with $\calb_{\calw^\flat}\ne\emptyset$. Such a book of
$I$-bundles $\calw^\flat$ is not normal, and the following proposition would
become false if the normal book of $I$-bundles $\calw$ were replaced
by $\calw^\flat$.

\Proposition\label{i once was as meek} Let $\calw$ be a normal 
book of $I$-bundles, and suppose that $A$ is an essential annulus in
$|\calw|$. Then $A$ is ambiently isotopic in $|\calw|$ to either a
vertical annulus in a page of $\calw$, or an annulus contained in a
binding.
\EndProposition

\Proof
Set $W=|\calw|$. Let $\calc$ be a regular neighborhood of $\cala_\calw$
in $W$. The definition of normality implies that (up to isotopy)
$\Sigma\doteq\overline{W-\calc}$ is the characteristic submanifold of
$W$ relative to $\partial W$. 

Let $V$ be a regular neighborhood of $A$ in $W$. Then
$V$ may be given the structure of a Seifert
fibered space  in such a way that $V\cap\partial W$ is a saturated
$2$-manifold in $\partial V$. The components of the  frontier of $V$
in $W$ are essential annuli in $W$. Hence the property (\ref{big ol'
hash}) of $\Sigma$ stated in \ref{oshkosh bighash} implies that $V$ is
ambiently isotopic in $W$ to a submanifold of $\Sigma$. In particular, $A$ is
ambiently isotopic in $W$ to an annulus $A'\subset
W-\cala_\calw$. Thus $A'$ is contained in either a page or a binding
of $W$.  

If $A'$ is contained in a page $P$, then $P$ is an $I$-bundle
  over a surface.  Since $A$ is essential in $W$ (relative to
  $\partial W$), it is in particular  essential in $P$ relative
  to $\partial_hP$. It therefore follows from Proposition \ref{one
    bundle at a time} that $A'$ is isotopic in $P$,
  by an ambient isotopy of $P$ which fixes $\partial_v P$, 
  to a vertical
  annulus. The conclusion of the proposition follows.
\EndProof

\Lemma\label{as a new-born lamb} Let $\calw$ be a normal
book of $I$-bundles. Let $Y$ be a compact $3$-dimensional submanifold
 of $ |\calw|$. Suppose that the following conditions hold.
\begin{enumerate}
\item\label{i'm now employed}Each
component of the frontier of $Y$ in $|\calw|$ is an essential properly
embedded annulus in
$|\calw|$.
\item\label{the devil may take him}
The $2$-manifold $Y\cap \partial |\calw|$ has two components $Z_0$ and
$Z_1$, with
$\chibar(Z_0)=\chibar(Z_1)=1$.
\item\label{i'll never forsake him}The inclusion homomorphism
$H_1(Z_0;\ZZ_2)\to H_1(Y;\ZZ_2)$ is injective. 
\item\label{the full treatment}For every solid torus
$L\subset Y$, such that $ L\cap Z_0$ is an annulus which is
homotopically non-trivial in $|\calw|$, the
inclusion homomorphism $H_1( L\cap Z_0;\ZZ)\to H_1(L,\ZZ)$ is 
surjective. 
\end{enumerate}
Then  the inclusion homomorphism
$\pi_1(Z_0)\to\pi_1(Y)$ is surjective. 
\EndLemma

\Proof 
We set $W=|\calw|$, $\calp=\calp_\calw$ and $\calb=\calb_\calw$.

Since the frontier components of $Y$
relative to $W$ are annuli, we have 
$$\chibar(Y)=\frac12\chibar(\partial Y)=\frac12(\chibar(Z_0)+\chibar(Z_1))=1.$$

By Proposition \ref{i once was as meek} we may assume that each
frontier component of $Y$ is either a vertical annulus contained in
$\calp $ and disjoint from $\cala $, or an annulus in
$\calb $. Then each component of $Y\cap\calp $ is an
$I$-sub-bundle of a page of $\calw$, whose frontier is a disjoint
union of vertical annuli in the page; each of these vertical annuli is
either contained in, or disjoint from, the vertical boundary of the
page. Furthermore, each component of $Y\cap\calb $ is a solid
torus in a binding, whose frontier is a disjoint union of
$\pi_1$-injective annuli in the binding. In particular we have
$\chibar(C)\ge0$ for every component $C$ of $Y\cap\calp $, and
$\chibar(C)=0$ for every component $C$ of $Y\cap\calb $.

Since the frontier components of $Y\cap\calb $ and $Y\cap\calp $
relative to $Y$ are annuli, we have
$$1=\chibar(Y)=\chibar(Y\cap\calb )+\chibar(Y\cap\calp
)=\chibar(Y\cap\calp )=\sum_U\chibar(U),$$
where $U$ ranges over the components of $Y\cap\calp  $.
Hence there is a component $U_1$ of $Y\cap\calp $ with
$\chibar(U_1)=1$, and $\chibar(U)=0$ for every component $U\ne U_1$ of
$Y\cap\calp $.

Since $U_1$ is a sub-bundle of a page of $\calw$, it is an $I$-bundle
over a surface $S$ with $\chibar(S)=1$, and the frontier of $U_1$
relative to $W$ is its vertical boundary. 

We set $\calk=\overline{Y-U_1}$.

Each component of the frontier of $U_1$ or $\calk$ relative to $W$ is
a component of either the frontier $\cala_\calw$ of $\calp_\calw$
relative to $W$, or the frontier of $Y$ relative to $W$.  According to
\cite[Lemma 5.2]{CDS},
the components of
$\cala_\calw$ are $\pi_1$-injective annuli in $W$. The components of
the frontier of $Y$ are $\pi_1$-injective by hypothesis. Hence:
\Claim\label{becomes perforce}
Each component of the frontier of $U_1$ or $\calk$ relative to $W$ is a
$\pi_1$-injective annulus in
$W$.
\EndClaim

The horizontal boundary of $U_1$ is a two-sheeted covering space of
$S$, which is connected if and only if $U_1$ is a twisted
$I$-bundle. In particular we have
$\chibar(\partial_hU_1)=2\chibar(S)=2$. On the other hand, we have
$\partial_hU_1= U_1\cap\partial W\subset Y\cap\partial W=Z_0\cup
Z_1$. Hence $\partial_hU_1\subset Z_j$ for some $j\in\{0,1\}$. It
follows from
\ref{becomes perforce}  that $\partial_hU_1$ is $\pi_1$-injective in $W$
and hence in $Z_0\cup Z_1$. Since $\chibar(Z_0)=\chibar(Z_1)=1$, the
surface $\partial_hU_1$ cannot be  contained in $Z_0$ or in $Z_1$. Hence:
\Claim\label{you'll never regret it}
$U_1$ is a trivial $I$-bundle over $S$, and its horizontal boundary
has one component contained in $Z_0$ and one in $Z_1$.
\EndClaim

For $i=1,2$, let us denote by $\Delta_i$ the component of
$\partial_hU_1$ contained in $Z_i$.

Since $W$ is simple by the definition of a normal book of
  $I$-bundles, it follows from
\ref{becomes perforce} that each component of $\calk$ is simple.
If $V$  is any  component  of $\calk$, the components of $V\cap\calp$
 are components of $Y\cap\calp$ distinct from $U_1$, and the components of $V\cap\calb$
 are components of $Y\cap\calb$. Hence all 
components of $V\cap\calp$ and $V\cap\calb$ have Euler characteristic
$0$. As the components of $(V\cap\calp)\cap(V\cap\calb)$ are annuli it
follows that $\chi(V)=0$. But the only simple $3$-manifold with
non-empty boundary having Euler characteristic $0$ is a solid torus.
This shows:
\Claim\label{i readily bet it}Every component of $\calk$ is a solid torus.
\EndClaim

Let us consider any component $A$ of $U_1\cap\calk$, and let $V$
denote the component of $\calk$ containing $A$. According to \ref{i
  readily bet it}, $V$ is a solid torus. Since $A$ is a component of
the frontier of $U_1$ in $W$, it is a $\pi_1$-injective annulus in $W$
by \ref{becomes perforce}, and hence in $V$. On the other hand, it
follows from \ref{you'll never regret it} that some component $c$ of
$\partial A$ is contained in $Z_0$. A small non-ambient
  isotopy of $V$ gives a solid torus $L$ such that 

$ L\cap Z_0$ is a regular neighborhood $R$ of $c$. Since $A$ is
$\pi_1$-injective in $W$, the annulus $R$ is
homotopically non-trivial in $W|$. Hypothesis  (\ref{the
    full treatment}) then implies that the 
inclusion homomorphism $H_1( L\cap Z_0;\ZZ)\to H_1(L,\ZZ)$ is an
isomorphism, and hence that the inclusion
homomorphism $\pi_1(A)\to\pi_1(V)$ is an isomorphism.
This shows:

\Claim\label{for duty, duty must be done}For every component $A$ of 
$U_1\cap\calk$, if
$V$ denotes the component of $\calk$ containing $A$, the inclusion
homomorphism $\pi_1(A)\to\pi_1(V)$ is an isomorphism.
\EndClaim

Now suppose that $A_1$ and $A_2$ are two components of $U_1\cap\calk$
contained in a single component $V$ of $\calk$. For $i=1,2$ it follows
from \ref{you'll never regret it} that some component $c_i$ of
$\partial A_i$ is contained in $Z_0$. Since $c_1$ and $c_2$ are
disjoint, non-trivial simple closed curves on the torus $\partial
V\subset W$, they represent the same element of $H_1(W;\ZZ_2)$.  Since
by hypothesis (\ref{i'll never forsake him}) the inclusion
homomorphism $H_1(Z_0;\ZZ_2)\to H_1(Y;\ZZ_2)$ is injective, $c_1$ and
$c_2$ cobound a subsurface $Q$ of $Z_0$. Since $c_1$ and $c_2$ are
homotopically non-trivial, $Q$ is $\pi_1$-injective in $Z_0$, and
hence $\chibar(Q)\le\chibar(Z_0)=1$. As $Q$ is orientable and has
exactly two boundary curves, it must be an annulus. On the other hand,
$Q$ and $\Delta_0$ are subsurfaces of $Z_0$, and we have $\partial
Q\subset\partial \Delta_0$. Hence either $\Delta_0\subset A$ or $A$ is
a component of $\overline{Z_0-\Delta_0}$. But it follows from
\ref{becomes perforce}  that $\Delta_0$ is $\pi_1$-injective in $W$
and hence in $Z_0$; since $\chibar(\Delta)=1$, we cannot have
$\Delta_0\subset A$. Thus we have proved:
\Claim\label{the rule applies}If
$A^{(0)}$ and $A^{(1)}$ are two components of $U_1\cap\calk$
contained in the same component of $\calk$, then some component  of
$\overline{Z_0-\Delta_0}$  is an annulus $Q$ having one boundary component
contained in $\partial A^{(0)}$ and one contained in $\partial A^{(1)}$.
\EndClaim

In particular, if $Q$ is an annulus having the properties stated in
 \ref{the rule applies}, then $\partial Q\subset
 \partial\Delta_0$. Since $\chibar(\Delta_0)=1$, the surface
 $\Delta_0$ has at most three boundary curves. Hence there is at most
 one unordered pair $\{A^{(0)},A^{(1)}\}$ of components of $U_1\cap\calk$ such
 that $A^{(0)}$ and $A^{(1)}$ are
contained in the same component of $\calk$. Equivalently:
\Claim\label{to everyone}
There is no component of $\calk$ whose boundary contains more than two
 components of $U_1\cap\calk$, and there is at most one component of
 $\calk$ whose boundary contains two components of $U_1\cap\calk$.
\EndClaim

In view of \ref{to everyone}, the argument now divides into two cases.

\noindent{\bf Case I: The boundary of each component of $\calk$ 
contains only one component of $U_1\cap\calk$.} In this case, let
$A_1,\ldots,A_m$ denote the components of $U_1\cap\calk$. (Since
$\chibar(S)=1$ we have $m\le3$.) Let $V_i$ denote the component of
$\calk$ containing $A_i$; thus the solid tori $V_1,\ldots,V_m$ are all
distinct. We have $Y=U_1\cup V_1\cup\cdots\cup V_m$. By 
\ref{for duty, duty must be done} the inclusion homomorphism
$\pi_1(A_i)\to\pi_1(V_i)$ is an isomorphism for $i=1,\ldots,m$. Hence 
the inclusion homomorphism 
$\pi_1(U_1)\to\pi_1(Y)$ is an isomorphism. But by \ref{you'll never regret
it}, the inclusion homomorphism 
$\pi_1(\Delta_0)\to\pi_1(U_1)$ is an isomorphism. Hence the inclusion
homomorphism $\pi_1(\Delta_0)\to\pi_1(Y)$ is an isomorphism, and in 
particular the inclusion
homomorphism $\pi_1(Z_0)\to\pi_1(Y)$ is surjective. This establishes
the conclusion of the lemma in this case.

\noindent{\bf Case II: There is a component $V_0$ of $\calk$ whose boundary
contains more than one component of $U_1\cap\calk$.} By \ref{to
everyone}, this component $V_0$ is unique and contains exactly two
components of $U_1\cap\calk$, which we denote $A_0^{(0)}$ and $A_0^{(1)}$. 

According to \ref{i readily bet it}, $V_0$ is a solid torus, and
according to \ref{for duty, duty must be done} the inclusion homomorphism
$\pi_1(A_0^{(i)})\to\pi_1(V_0)$ is an isomorphism for $i=0,1$. Hence
there is a homeomorphism $h:V_0\to S^1\times[0,1]\times[0,1]$ such
that $h(A_0^{(i)})=S^1\times\{i\}\times[0,1]$ for $i=0,1$. This allows
us to extend the $I$-bundle structure of $U_1$ to an $I$-bundle
structure for $L\doteq U_1\cup V_0$. (Note, however, that the
horizontal boundary of $L$ need not be contained in $\partial W$.) 

According to \ref{the rule applies},
some component  of
$\overline{Z_0-\Delta_0}$  is an annulus $Q$ having one boundary component
contained in $\partial A^{(0)}$ and one contained in $\partial
A^{(1)}$. This implies that the $I$-bundle $L$ is trivial, and that
one of its horizontal boundary components, which we shall denote by
$\Theta$, is contained in $Z_0$.

Now let $A_1,\ldots,A_m$ denote the components of $U_1\cap\calk$
distinct from $A_0^{(0)}$ and $A_0^{(1)}$. (Since $\chibar(S)=1$ one
may show that $m\le1$, but this will not be used.) Let $V_i$ denote
the component of $\calk$ containing $A_i$; thus the solid tori
$V_1,\ldots,V_m$ are all distinct. We have $Y=L\cup
V_1\cup\cdots\cup V_m$. By
\ref{for duty, duty must be done} the inclusion homomorphism
$\pi_1(A_i)\to\pi_1(V_i)$ is an isomorphism for $i=1,\ldots,m$. Hence 
the inclusion homomorphism 
$\pi_1(L)\to\pi_1(Y)$ is an isomorphism. But since $L$ is a
trivial $I$-bundle and $\Theta$ is one component of its horizontal boundary, the inclusion homomorphism 
$\pi_1(\Theta)\to\pi_1(L)$ is an isomorphism. Hence the inclusion
homomorphism $\pi_1(\Theta)\to\pi_1(Y)$ is an isomorphism, and in 
particular the inclusion
homomorphism $\pi_1(Z_0)\to\pi_1(Y)$ is surjective. Thus 
the conclusion of the lemma is established in this case as well.
\EndProof


\section{Trimonic  manifolds}  \label{sec:trimonic}

\Number\label{troglodyte}
Let $V$ be a point of an oriented surface $S$. By an {\it ordered
  triod} based at $V$ we shall mean an ordered triple $(A_0,A_1,A_2)$
of closed topological arcs in $S$, each having $V$ as an endpoint,
such that $A_i\cap A_j=\{V\}$ whenever $i\ne j$. 

Suppose that $(A_0,A_1,A_2)$ is an ordered triod based at $V$. For
$i=0,1,2$ let $x_i$ denote the endpoint of $A_i$ that is distinct from
$V$. Then there is a disk $\delta \subset S$ such that $A_0\cup
A_1\cup A_2\subset \delta $ and $A_0\cup A_1\cup A_2\cap\partial
\delta =\{x_0,x_1,x_2\}$. We shall express this by saying that the
triod $(A_0,A_1,A_2)$ is {\it properly embedded} in $\delta$. The
orientation of $S$ restricts to an orientation of $\delta $, which in
turn induces an orientation of $\partial \delta $. We shall say that
the ordered triod $(A_0,A_1,A_2)$ is {\it positive} if the ordered
triple $(x_0,x_1,x_2)$ is in counterclockwise order on $\partial
\delta $, and {\it negative} otherwise.
\EndNumber

\Number Suppose that $\theta$ is an oriented open arc. We
denote by $\theta'$ the same arc with the opposite orientation. By a
{\it terminal segment} of $\theta$ we mean a subset $A$ of $\theta$ which
has the form $h((t,1))$ for some orientation-preserving homeomorphism
$h:(0,1)\to \theta$ and some point $t\in(0,1)$.  By an {\it initial
segment} of $ \theta$ we mean a terminal segment of $\myprime \theta$.
\EndNumber

\Number\label{terminator} Now suppose that $\Gamma$ is a graph (i.e. a $1$-dimensional
CW complex) contained in an oriented surface $S$. By an {\it oriented
edge} of $\Gamma$ we mean simply an (open) edge which is equipped with
an orientation. Let $V$ be a vertex of $\Gamma$, and let
$(e_0,e_1,e_2)$ be an ordered triple of oriented edges of $\Gamma$
with terminal vertex $V$. Assume that the $e_i$ are distinct as
oriented edges (although two of them may be opposite orientations of
the same underlying edge). We may choose terminal segments $A_i$ of
the $e_i$ in such a way that $(\bar{A}_0,\bar{A}_1,\bar{A}_2)$ is an ordered triod in
$S$. We shall say that the ordered triple $(e_0,e_1,e_2)$ is {\it
positive} if the ordered triod $(\bar{A}_0,\bar{A}_1,\bar{A}_2)$ is positive, and {\it
negative} otherwise.
\EndNumber

\Number
Let $Z$ be a planar surface with three boundary curves, and let
$\star\in\inter Z$ be a base point. An ordered pair 
$(z_1,z_2)$ of elements of $\pi_1(Z,\star)$ will be called a {\it geometric
basis} for $\pi_1(Z,\star)$ if the boundary curves of $Z$ may be
indexed as $(C_i)_{1\le i\le3}$ in such a way that for $i=1,2,3$ there exist a point
$p_i\in C_i$, a closed path $\gamma_i$ in $C_i$ based at $p_i$ and an
oriented (embedded) arc $\tau_i$ from $\star$ to $p_i$, such that
\begin{itemize}
\item the  $\tau_i$ have pairwise disjoint interiors;
\item $\tau_i\cap C_i=\{p_i\}$ for each $i$;
\item $[\gamma_i]$ generates $\pi_1(C_i,p_i)$ for each $i$; 
\item $z_i=[\tau_i*\gamma_i*\overline{\tau_i}]$ for $i=1,2$; and
\item $z_1^{-1}z_2= [\tau_3*\gamma_3*\overline{\tau_3}]$.
\end{itemize}

Note that if $(z_1,z_2)$ is a geometric
basis for $\pi_1(Z,\star)$ then $z_1$ and $z_2$ freely generate
$\pi_1(Z,\star)$.
\EndNumber

\Number\label{but go} 
Let $\Gamma$ be a theta graph contained in an oriented surface
$S$. Let $W$ and $V$ denote the vertices of $\Gamma$, and let
$\beta_0$, $\beta_1$ and $\beta_2$ denote the oriented edges having
initial vertex $W$ and terminal vertex $V$. Suppose that the ordered
triple $(\beta_0,\beta_1,\beta_2)$ of edges terminating at $V$ is
positive, and that the ordered triple $(\myprime \beta_0,\myprime
\beta_1,\myprime \beta_2)$ of oriented edges terminating at $W$ is
negative. Then a regular neighborhood $Z$ of $\Gamma$ in $S$ is a
planar surface with three boundary curves, and $([\myprime
\beta_0* \beta_1],[\myprime \beta_0* \beta_2])$ is a geometric
basis of $\pi_1(Z,V)$.
\EndNumber

\Number\label{go in peace}
Let $\Gamma$ be an eyeglass graph contained in an oriented surface
$S$. Let $W$ and $V$ denote the vertices of $\Gamma$, let $\beta_0$
denote the oriented edge having initial vertex $W$ and terminal
vertex $V$, and let $\beta_1$ and $\beta_2$ be oriented loops based at
$W$ and $V$ respectively.  Suppose that the ordered triple
$(\beta_0,\beta_2',\beta_2)$ of edges terminating at $V$ is positive,
and that the ordered triple $(\beta_0',\beta_1',\beta_1)$ of edges
terminating at $W$ is negative. Then a regular neighborhood $Z$ of
$\Gamma$ in $S$ is a planar surface with three boundary curves, and
$([ \beta_2],[\myprime \beta_0*
\beta_1*\beta_0])$ is a geometric basis of $\pi_1(Z,V)$.
\EndNumber

\Definition\label{i'm now sir murgatroyd}
Let $X$ be a compact orientable $3$-manifold, and let $S$ be a
component  of $\partial X$. We shall say that $X$ is a {\it trimonic 
 manifold relative to $S$} if there exists a properly
embedded arc $\alpha\subset X$ and a PL
map $f$ of a PL $2$-disk $D$ into $X$, such that the following conditions hold:
\begin{enumerate}
\item\label{of steven sondheim} $f^{-1}(\alpha)$ is a union of three disjoint arcs in $\partial
  D$;
\item\label{and leonard bernstein} $f$ maps each component of $f^{-1}(\alpha)$ homeomorphically
  onto $\alpha$;
\item\label{everything's free} $f|(\inter D\cup((\partial D)-f^{-1}(\alpha)))$ is one-to-one;
\item\label{in america} $f(\inter D)\subset\inter X$; 
\item\label{for a small fee} $f((\partial D)-f^{-1}(\alpha))\subset S$;
\item\label{maybe i go back} $X$ is a semi-regular neighborhood of $S\cup f(D)$.
\end{enumerate}

Note that condition (\ref{for a small fee}) implies that the endpoints
of $\alpha$ lie in $S$.

A PL map $f$ of a $2$-disk $D$ into $X$ such that (\ref{of steven
sondheim})--(\ref{maybe i go back}) hold for some properly embedded
arc $\alpha$ in $X$ will be called a {\it
defining hexagon} for the trimonic  manifold $X$ relative to
$S$.

\EndDefinition

\Notation
Suppose that $f:D\to X$ is a defining hexagon for a trimonic 
manifold $X$ relative to $S$. Note that the arc $\alpha$ appearing in
Definition \ref{i'm now sir murgatroyd} is uniquely determined by
$D$. We shall denote this arc by $\alpha_f$. Furthermore, we shall
denote by $\Gamma_f$ the PL set $f({(\partial
D)-f^{-1}(\alpha)})=f(D)\cap\partial X\subset S$.
\EndNotation

\Lemma\label{the magic} Suppose that $X$ is a trimonic
 manifold relative to $S$. Let $f:D\to X$ be a defining
hexagon for $X$, and set $\alpha=\alpha_f$ and $\Gamma=\Gamma_f$. Then
$\Gamma$ is homeomorphic to either a theta graph 
(the ``theta case") or an eyeglass graph
(the ``eyeglass case"),
and a regular neighborhood $Z$ of $\Gamma$ in $S$ is a planar surface
with three boundary curves. Furthermore, for some (and hence for any)
base point $\star\in\inter Z$, there is an ordered  basis
$(t,u_1,u_2)$ of the rank-$3$ free group $\pi_1(Z\cup\alpha,\star)$
with the following properties:
\begin{itemize}
\item the  inclusion homomorphism
$\pi_1(Z,\star)\to\pi_1(Z\cup\alpha,\star)$ maps some geometric basis
of $\pi_1(Z,\star)$ to the pair $(u_1,u_2)$;
 and
\item $\partial D$ may be oriented so that the conjugacy class in
$\pi_1(Z\cup\alpha,\star)$ represented by $f|\partial D:\partial D\to
Z\cup\alpha$ is $t^2u_1tu_2$ in the theta case, 
or $t^2u_1t^{-1}u_2$ in the eyeglass case.
\end{itemize}
\EndLemma

\begin{figure}
\begin{center}
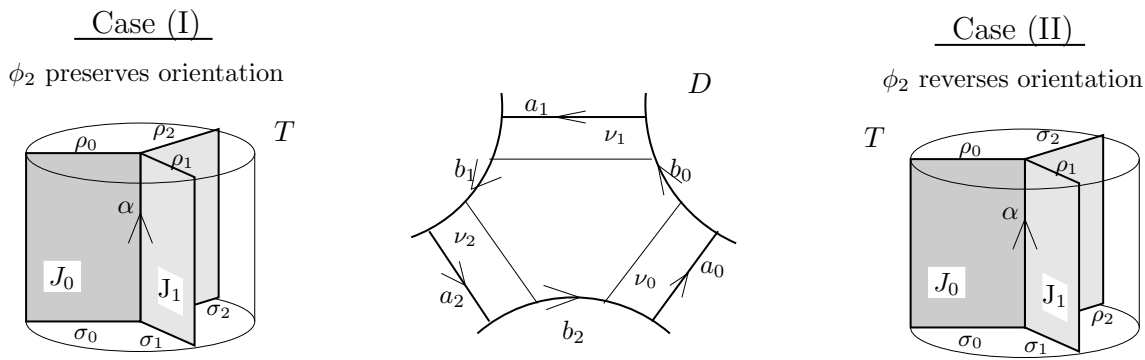
\end{center}
\caption{Some objects defined in the proof of Lemma \ref{the magic}}
\label{hexnT}
\end{figure}

\Proof
We denote the components of
$f^{-1}(\alpha)$ by $a_0$, $a_1$ and $a_2$, and we denote the
components of $\overline{(\partial D)-f^{-1}(\alpha)}$ by $b_0$, $b_1$
and $b_2$. We take the $a_i$ and $b_i$ to be labeled in such a way
that $a_i$ and $b_j$ share an endpoint if and only if $i$ is congruent
to either $j$ or $j+1$ modulo $3$. For $i=0,1,2$ we set
$\beta_i=f(\inter b_i)$. Then $\beta_i$ is an open arc in
$S-\partial\alpha $, and
$\overline{\beta_i}-\beta_i\subset\partial\alpha $. Hence
$\Gamma=\overline{\beta_0\cup\beta_1\cup\beta_2}$ may be given the
structure of a graph whose vertices are the
endpoints of $\alpha$, and whose edges are $\beta_0$, $\beta_1$ and
$\beta_2$.

Since $\alpha$ is an embedded arc, $f$ maps the terminal point of
$\beta_i$ and the initial point of $\beta_{i+1}$ to distinct vertices
of $\Gamma$. In particular, each vertex of $\Gamma$ has 
valence $3$.  Since $\Gamma$
has three edges and two vertices, it is either a theta graph
or an eyeglass graph.

For $i=0,1,2$ we denote by $P_i$ the common endpoint of the arcs $a_i$
and $b_{i-1}$ in $\partial D$ (where subtraction is interpreted modulo
$3$), and by $Q_i$ the common endpoint of $a_i$ and $b_{i}$.

If we fix an orientation of $\partial D$ then each of the arcs $a_i$
and $b_i$ inherits an orientation. We choose the orientation of
$\partial D$ in such a way that $P_i$ and $Q_i$ are respectively the
initial and terminal points of $a_i$, while $Q_i$ and $P_{i+1}$ are
respectively the initial and terminal points of $b_i$.
According to Definition \ref{i'm now sir murgatroyd}, $f$ restricts to
a homeomorphism $\phi_i:a_i\to\alpha$ for $i=0,1,2$.

Let $T$ be a PL tubular neighborhood of $\alpha$ in $X$. We may choose
$T$ in such a way that $f$ is transverse to the frontier of $T$ and
$f^{-1}(T)$ is a regular neighborhood of $f^{-1}(\alpha)=a_0\cup
a_1\cup a_2$ in $D$. For each $i\in\{0,1,2\}$, let $\nu_i$ denote the component of $f^{-1}(T)$
containing $a_i$. Then $\nu_i\cap\partial D$ has the form
$a_i\cup s_i\cup r_i$, where $ s_i$ is the closure of a terminal segment
of $b_{i-1}$ and $ r_i$ is the closure of an initial segment of $b_i$.

It follows from Definition \ref{i'm now sir murgatroyd} that $f$ maps
  each $\nu_i$ homeomorphically on to a PL disk $J_i\subset T$, and
  that $J_i\cap J_j=\alpha$ when $i\ne j$. The intersection of $T$
  with $\partial X$ consists of two PL disks; for each
  $i\in\{0,1,2\}$, one of these disks meets $J_i$ in the arc
  $\sigma_i\doteq f( s_i)$, and  the other meets $J_i$ in the arc
  $\rho_i\doteq f( r_i)$. Hence we may identify $T$ by a PL homemorphism
  with $\delta\times\alpha$, where $\delta$ is a PL disk, in such a
  way that $\alpha=\{o\}\times\alpha$ for some interior point $o$ of
  $\delta$; and so that for $i=0,1,2$ we have $J_i=t_i\times\alpha$
  for some arc $t_i\subset\delta$. Each $t_i$ has one endpoint in
  $\partial\delta$ and one at $o$, and we have $t_i\cap t_j=o$ for
  $i\ne j$. The components of $t_i\times\partial\alpha$ are $\sigma_i$ and
  $\rho_i$.

We may orient $\alpha$ in such a way that at
least two of the homeomorphisms $\phi_i:a_i\to\alpha$ are
orientation-preserving. Hence, after a possible cyclic relabeling of
the $a_i$ (and the $b_i$), we may assume that $\phi_0$ and $\phi_1$ are
orientation-preserving. 

We denote by $V$ and $W$, respectively, the initial and
terminal endpoints of $\alpha$ with respect to the orientation that we
have chosen.  Thus for $i=0,1$ we have $f(P_i)=V$ and $f(Q_i)=W$.

We shall distinguish two cases according to whether the homeomorphism
$\phi_2:a_2\to\alpha$ (I) preserves or (II) reverses orientation.  In
case I we have $f(P_2)=V$ and $f(Q_2)=W$, while in case II we have
$f(P_2)=W$ and $f(Q_2)=V$. Hence we may define ordered triods (see
\ref{troglodyte}) $\calt_V$ and $\calt_W$ based at $V$ and $W$
respectively by setting $\calt_V=(\sigma_0,\sigma_1,\sigma_2)$ and
$\calt_W=(\rho_0,\rho_1,\rho_2)$ in case I, and
$\calt_V=(\sigma_0,\sigma_1,\rho_2)$ and
$\calt_W=(\rho_0,\rho_1,\sigma_2)$ in case II.  In each case, if we
denote by $\delta_V$ and $\delta_W$ the components of $T\cap\partial
X$ containing $V$ and $W$ respectively, the triods $\calt_V$ and
$\calt_W$ are properly embedded (see \ref{troglodyte}) in $\delta_V$
and $\delta_W$.

If we identify $T$ as above with $\delta\times\alpha$, we have
$\delta_V=\delta\times\{V\}$ and $\delta_W=\delta\times\{W\}$. We may
then define a homeomorphism $\psi:\delta_V\to\delta_W$ by
$\psi(x,V)=(x,W)$. 

We orient $S$ in such a way that $\calt_V$ is a positive
ordered triod (see \ref{troglodyte}). Each of the disks $\delta_V$
and $\delta_W$
inherits an orientation from $S$. The orientability of the
$3$-manifold $X$ implies that $\psi:\delta_V\to\delta_W$ is an
orientation-reversing homeomorphism.

By the properties of the identification stated above, we have
$\psi(\sigma_i)=\rho_i$ for $i=0,1$; and in Case I we have
$\psi(\sigma_2)=\rho_2$, while in Case II we have
$\psi(\rho_2)=\sigma_2$. Since $\calt_V$ is a positive ordered triod,
and since $\psi$ reverses orientation, the ordered triod $\calt_W$ is
negative.

We orient each open arc $\beta_i$ in such a way that $f|\inter
b_i:\inter b_i\to\beta_i$ is orientation-preserving.  Since $\alpha$,
$\beta_0$, $\beta_1$ and $\beta_2$ are now equipped with orientations,
their closures define elements  of the fundamental groupoid 
$\Pi(f(D))$ which we denote by $[\alpha]$, $[\beta_0]$,
$[\beta_1]$ and $[\beta_2]$.  We let $c$ denote the closed path
$a_0*b_0*a_1*b_1*a_2*b_2$ in $\partial D$, based at $P_0$.  Then $[c]$
generates $\pi_1( \partial D,P_0)$. We set $\gamma= f\circ c$.

We
have
$$[\gamma] =[\alpha][\beta_0][\alpha][\beta_1][\alpha]^\epsilon[\beta_2]
\in\pi_1(f(D),V)\subset\Pi(f(D)),$$
 where $\epsilon=1$ in Case I, and $\epsilon=-1$ in Case II.

To prove the conclusions of the lemma in Case I, we note that since
$f(P_i)=V$ and $f(Q_i)=W$ for each $i$, each $\beta_i$ has $W$ as
initial vertex and $V$ as terminal vertex. Thus $\Gamma$ is a theta
graph, and Case I is the ``theta case'' referred to in the 
statement of the Lemma. 

Since $\calt_V=(\sigma_0,\sigma_1,\sigma_2)$ is a positive
ordered triod based at $V$, and the interior of $\sigma_i$ 
is a terminal segment of
$\beta_{i-1}$, the triple $(\beta_2,\beta_0,\beta_1)$ of edges
terminating at $V$ is positive in the sense of \ref{terminator}. Hence
the triple $(\beta_0,\beta_1,\beta_2)$ is positive. Likewise, since
$\calt_W=(\rho_0,\rho_1,\rho_2)$ is a negative ordered triod based at
$W$, and since the interior of $\rho_i$ is an 
initial segment of $\beta_i$---and
therefore a terminal segment of $\beta_i'$---the triple
$(\beta_0',\beta_1',\beta_2')$ of edges terminating at $W$ is
negative.  It now follows from \ref{but go} that a regular
neighborhood $Z$ of $\Gamma$ in $S$ is a planar surface with three
boundary curves, and that if we set $z_1=[\myprime
\beta_0* \beta_1]$ and $z_2=[\myprime \beta_0* \beta_2])$, then $(z_1,z_2)$
is a geometric basis of $\pi_1(Z,V)$. 

In particular $\pi_1(Z,V)$ is freely generated by $z_1$ and $z_2$, and
hence $\pi_1(Z\cup\alpha,V)$ is generated by $t$, $u_1$ and $u_2$,
where $u_i$ denotes the image of $z_i$ under the inclusion
homomorphism $\pi_1(Z,V)\to\pi_1(Z\cup\alpha,V)$, and
$t=[\alpha*\beta_0]$.

We have
$$\begin{aligned}
\empty[ \gamma ]
&=[\alpha*\beta_0*\alpha*\beta_1*\alpha*\beta_2]\\
&=[\alpha*\beta_0*\alpha*\beta_0*\myprime\beta_0*\beta_1*\alpha*\beta_0*\myprime\beta_0*\beta_2]\\
&=t^2u_1tu_2.
\end{aligned}
$$

This establishes the conclusions of the Lemma in the theta case.

In Case II we have $f(P_i)=V$ and $f(Q_i)=W$ for $i=0,1$, while
$f(P_2)=W$ and $f(Q_2)=V$. It follows that the closures of
$\beta_1$ and $\beta_2$ are
loops based at $W$ and $V$ respectively, whereas $\beta_0$ has initial
point $W$ and terminal point $V$. Thus $\Gamma$ is an
eyeglass graph, and Case II is the ``eyeglass case'' referred 
to above.

Since $\calt_V=(\sigma_0,\sigma_1,\rho_2)$ is a positive ordered triod
based at $V$, and since the interiors of
$\sigma_0$, $\sigma_1$ and $\rho_2$ are
respectively terminal segments of $\beta_{2}$, $\beta_0$ and
$\beta_2'$, the triple $(\beta_2,\beta_0,\beta_2')$ of edges
terminating at $V$ is positive.  Hence the triple
$(\beta_0,\beta_2',\beta_2)$ is positive.  Likewise, since
$\calt_W=(\rho_0,\rho_1,\sigma_2)$ is a negative ordered triod based
at $W$, and since the interiors of
$\rho_0$, $\rho_1$ and $\sigma_2$ are respectively
terminal segments of $\beta_{0}'$, $\beta_1'$ and $\beta_1$, the
triple $(\beta_0',\beta_1',\beta_1)$ of edges terminating at $W$ is
negative.  It now follows from \ref{go in peace} that a regular
neighborhood $Z$ of $\Gamma$ in $S$ is a planar surface with three
boundary curves, and that  if we set $z_1=[\myprime \beta_0*
\beta_1*\beta_0])$ and 
$z_2=[\beta_2]$, then $(z_1,z_2)$ is a geometric basis of
$\pi_1(Z,V)$.

In particular $\pi_1(Z,V)$ is freely generated by $z_1$ and $z_2$, and
hence $\pi_1(Z\cup\alpha,V)$ is generated by $t$, $u_1$ and $u_2$,
where $u_i$ denotes the image of $z_i$ under the inclusion
homomorphism $\pi_1(Z,V)\to\pi_1(Z\cup\alpha,V)$, and
$t=[\alpha*\beta_0]$.

We have
$$\begin{aligned}
\empty [\gamma ]
&=[\alpha*\beta_0*\alpha*\beta_1*\myprime\alpha*\beta_2]\\
&=[\alpha*\beta_0*\alpha*\beta_0*\myprime\beta_0*\beta_1*\beta_0*\myprime\beta_0*\myprime\alpha*\beta_2]\\
&=(t)^2u_1(t)^{-1}u_2.
\end{aligned}
$$
This establishes the conclusions of the lemma in the eyeglass case. 
\EndProof

\Definition \label{non-degenerate} Let $X$ be a  trimonic 
  manifold relative to $S$. We shall say that $X$ is {\it non-degenerate} if there is a defining hexagon $f:D\to X$
 such that no component of   $S-\Gamma_f$ is an open disk.  
\EndDefinition

\Lemma\label{as steward} Suppose that $X$ is a non-degenerate 
trimonic  manifold 
relative to a component $S$ of  $\partial X$.
Then 
$Hg(X)\le1+\genus (S)$. Furthermore, 
$X$ contains a
compact connected $3$-dimensional submanifold  $Y$ with the following
properties. 
\begin{enumerate}
\item\label{for greater precision} Each
component of the frontier of $Y$ in $X$ is an essential annulus in $X$, 
joining distinct components of $\partial X$.
\item\label{without the elision} The $2$-manifold $Y\cap \partial X$
has two components $Z_0$ and $Z_1$, where $Z_0\subset S$ and
$Z_1\not\subset S$, and each $Z_i$ is a planar surface with
three boundary curves.
\item\label{and i who was} The group $\pi_1(Y)$ is free of rank $2$.
\item\label{his valet de sham} 
For any base point $\star \in\inter Z_0$, there exists an
ordered pair of
generators $(x,y)$ of $\pi_1(Y,\star )$, such that either the pair
$(x,y x^{-1}y^2)$ or the pair $(x,y^{-1} x^{-1}y^2)$ is the image of
some geometric basis of $\pi_1(Z_0,\star )$ under the inclusion
homomorphism $\pi_1(Z_0,\star )\to\pi_1(Y,\star )$.
\item\label{oh baby}Each component $Q$ of $\overline{X-Y}$ may be
given the structure of a trivial $I$-bundle over a $2$-manifold, in
such a way that $Y\cap Q$ is the vertical boundary of $Q$.
\end{enumerate}
\EndLemma

\Proof
We fix a defining hexagon $f:D\to X$, and set $\alpha=\alpha_f$ and
$\Gamma=\Gamma_f$. Since $X$ is non-degenerate, we may choose $f$ in
such a way that:
\Claim\label{when an innocent heart}
No component of $S-\Gamma$ is an open disk.
\EndClaim

Let $T$ be a PL tubular neighborhood of $\alpha$ in $X$, and let
$\delta$ be a properly embedded disk in $T\cap\inter X$ which crosses
$\alpha$ transversally in one point. We may choose $T$ in such a way
that $f^{-1}(T)$ is a regular neighborhood of $f^{-1}(\alpha)=a_0\cup
a_1\cup a_2$ in $D$. Let $D'$ denote the disk
$\overline{D-f^{-1}(T)}$. Then the frontier of $D'$ is the union of
three arcs $a_0'$, $a_1'$ and $a_2'$, where $a_i'$ is the frontier in
$D$ of the component of $f^{-1}(T)$ containing $a_i$.

According to Definition \ref{i'm now sir murgatroyd}, $f$ maps each
$a_i$ homeomorphically onto $\alpha$. Hence we may choose $\delta$ so
that for each $i$, the arc $f(a_i')$
meets $\partial\delta$ transversally in 
exactly one point. In particular:
\Claim\label{your father takes bubble baths}
The simple closed curve $f(\partial D')$ meets $\partial\delta$
 in exactly three points, and these are transversal points of
 intersection within the frontier annulus of $T$.
\EndClaim

Let $\eta$ denote a regular neighborhood in $\overline{X-T}$ of the
properly embedded disk $f(D')\subset\overline{X-T}$. Then $K\doteq T\cup\eta$ is
a regular neighborhood of $f(D)$ in $X$.

Let $X'$ denote the manifold obtained from $X$ by attaching a collar
$\calc$ along the boundary component $S$. We identify
$\calc$ with $\Sigma\times I$, where $\Sigma$ is a surface
homeomorphic to $S$, in such a way that
$S=\Sigma\times\{1\}$ and $S'\doteq\Sigma
\times\{0\}\subset\partial X'$. We set $X''=\calc\cup
K\subset X'$. Since $K$ is a regular
neighborhood of $f(D)$ in $X$, the manifold $X'$ is a semi-regular
neighborhood of its $3$-dimensional submanifold $X''$. It follows that
the pair $(X'',S')$ is homeomorphic to
$(X',S')$, and hence to $(X,S)$. It therefore suffices to
show that the conclusions of the lemma are true when $X$ and $S$ are
replaced by $X''$ and $S'$.

We first note that the manifold $V \doteq(\Sigma\times I)\cup T$ is a
compression body with $\partial_- V=S'$. If we set
$g=\genus (S)$, the genus of $\partial_+ V$ is $g+1$. In particular
$V$ admits a Heegaard splitting of genus $g+1$. We have
$X''=V\cup\eta$, so that $X''$ is obtained from $V$ by adding
a $2$-handle. It therefore follows from \cite[\easyHeegaard]{CDS} that  $\Hg(X'')\le g+1$. This gives the first
assertion of the lemma.

We must now construct a compact connected $3$-dimensional submanifold
$Y$ of $X''$ such that Properties (\ref{for greater
precision})--(\ref{oh baby}) hold for $Y$ when 
 $X$ and $S$
are replaced by $X''$ and
$S'$. 

We take $Z$ to be a regular neighborhood of $K\cap S$ in $S$,
and observe that $Z$ is a
regular neighborhood of $f(D)\cap S=\Gamma$ in $S$. It therefore
follows from Lemma \ref{the magic}  that $Z$ is a planar surface
with three boundary curves.

We have
$Z=R\times\{1\}$ for some $R\subset\Sigma$. We set $Y=(R\times
I)\cup K\subset(\Sigma \times I)\cup K=X''$.

By construction we have $\overline{X-Y}=\overline{\Sigma-R}\times I$
and $\overline{X-Y}\cap(S')=\overline{\Sigma-R}\times
\{0\}$. This implies Property (\ref{oh baby}) for $Y$.

To show that $Y$ has Property (\ref{without the elision}), we first
observe that $Y=(R\times I)\cup K=(R\times I)\cup T\cup\eta$, and that
$T$ and $\eta$ are disjoint from $S'=\Sigma\times\{0\}$ (since their
intersection with $\Sigma\times I$ is contained in
$\Sigma\times\{1\}$). Hence $Y\cap\partial X''$ is the disjoint union
of $Z_0\doteq R\times\{0\}$ and $Z_1\doteq ((R\times \{1\})\cup
T\cup\eta)\cap\partial X''\not\subset S'$. The $2$-manifold $Z_0$ is
homeomorphic to $Z$, and is therefore a planar surface with three
boundary curves.

To describe the $2$-manifold $Z_1$, we first consider the $2$-manifold
$Z^*\doteq((R\times \{1\})\cup T)\cap\partial((\Sigma\times I)\cup
T)$. We may obtain $Z^*$ from $R\times\{1\}$ by removing the interior
of $T\cap(R\times\{1\})$, which is a  union of two disjoint disks, and
attaching the annulus $\overline{(\partial
T)-(T\cap(R\times\{1\}))}$. Since $R\times\{1\}\cong R$ is connected and
$\chibar(R\times\{1\})=1$, it follows that
$Z^*$ is connected and that
$\chibar(Z^*)=3$. The surface $Z^*$ contains the simple closed curves
 $f(\partial D')$ and $\partial\delta$, which by
\ref{your father takes bubble baths} meet transversally in exactly
three points. In particular, their mod $2$ homological intersection
number in $Z^*$ is equal to $1$. Hence $f(\partial D')$ does not
separate the connected surface $Z^*$. The $2$-manifold $Z_1$ is
obtained from $Z^*$ by removing the interior of the
annulus $Z^*\cap\eta$ and
attaching the two components of
$\overline{(\partial\eta)-(Z^*\cap\eta)}$, which are disks. Since
$\eta$ is a regular neighborhood of $f(D')$ in $\overline{X-T}$, the
annulus $Z^*\cap\eta$ is a regular neighborhood of the non-separating
curve $f(\partial D')$ in $T^*$. Hence $Z_1$ is connected, and so
$Z_0$ and $Z_1$ are the components of $Y\cap\partial X''$.
Furthermore, we have $\partial Z_1=\partial(R\times\{1\})$, and since
$R\cong Z$ is a planar surface with three boundary curves, $Z_1$ has three
boundary curves. Since in addition we have
$\chibar(Z_1)=\chibar(Z^*)-2=1$, the surface $Z_1$ must be planar.

This establishes Property (\ref{without the elision}) for $Y$.

We now turn to the verification of Properties (\ref{and i who was})
and (\ref{his valet de sham}). By construction the pair $(Y,Z_0)$ is
homotopy equivalent to $(Z\cup f(D),Z)$.  Hence it suffices to
show that if $\star\in\inter Z$ is a base point, 
then $\pi_1(Z\cup f(D),\star )$ is
free of rank $2$ and has an ordered basis $(x,y)$, such that the image of
some geometric basis of $\pi_1(Z,\star )$ under the inclusion
homomorphism is either $(x,y^2xy)$ or $(x,y^2xy^{-1})$.

We fix an ordered basis $(t,u_1,u_2)$ of the rank-$3$ free 
group $\pi_1(Z\cup\alpha,\star)$ having the properties stated in the
conclusion of Lemma \ref{the magic}.  In particular, if $\star\in
\inter Z$ is any base point,  $\pi_1(Z\cup\alpha,\star)$ is free on
the generators $t$, $u_1$ and $u_2$. Furthermore, $Z\cup f(D)$ is
obtained from $Z\cup\alpha$ by attaching a $2$-cell, and the attaching
map realizes either the conjugacy class of $t^2u_1tu_2$ or that of
$t^2u_1t^{-1}u_2$ in $\pi_1(Z\cup\alpha,\star)$. Hence $\pi_1(Z\cup
f(D),\star)$ is given by either the presentation
\Equation\label{a boring young reverend}
|t,u_1,u_2:t^2u_1tu_2=1|,
\EndEquation
in the theta case, or the presentation
\Equation\label{they thought he would never end}
|t,u_1,u_2:t^2u_1t^{-1}u_2=1|,
\EndEquation
in the eyeglass case.

Let $\bar t$ and $\bar u_i$ denote the respective images in
$\pi_1(Z\cup f(D),\star)$ of the generators $t$ and $u_i$ of the free
group on $t$, $u_1$ and $u_2$. From the properties of the basis 
$(t,u_1,u_2)$ stated in the conclusion of Lemma \ref{the magic}, it
follows that $u_1$ and $u_2$ are the images of the elements of a
geometric basis of $\pi_1(Z,\star)$ under the inclusion homomorphism
$\pi_1(Z,\star)\to\pi_1(Z\cup f(D),\star)$.  On the other hand, it is
clear from the presentation (\ref{a boring young
reverend}) or (\ref{they thought he would never end}) that
$\pi_1(Z\cup f(D),\star)$ is free on the generators $x\doteq \bar{u}_1$ and
$y\doteq \bar{t}^{-1}$. Furthermore, in the theta case,
we have $\bar u_1=x$ and $\bar u_2=y x^{-1}y^2$; while in the
eyeglass case, 
we have $\bar u_1=x$ and
$\bar u_2=y^{-1} x^{-1}y^2$. Thus properties (\ref{and i who was}) and
(\ref{his valet de sham}) of $Y$ are established.

To prove that $Y$ has Property (\ref{for greater precision}), we first
observe that by construction the frontier of $Y$ is $\partial R\times
I$. Hence each component $A$ of the frontier has the form $c\times I$,
where $c$ is a component of $\partial R$. Thus $A$ is an annulus
having one boundary curve in the component $\Sigma
\times\{0\}$ of $\partial X''$. The other boundary curve of $A$ is
contained in $\Sigma\times\{1\}$ and is therefore disjoint from
$S'=\Sigma\times\{0\}$. In particular, $A$ has its boundary curves in
distinct components of $\partial X''$. To prove that $A$ is essential,
it therefore suffices to prove that it is $\pi_1$-injective in
$X''$. This in turn reduces to showing that $A$ is $\pi_1$-injective
(\ref{ridin'}) in $Y$ and in $\overline{X''-Y}$.

To prove $\pi_1$-injectivity of $A$ in $Y$, we observe that since the
surface $R$ is homeomorphic to a regular neighborhood of
$\Gamma$ in $S$, we have $\chibar(R)=1$. Hence the component $c$ of
$\partial R$ is $\pi_1$-injective in $R$. This implies that
$c\times\{0\}$ is $\pi_1$-injective in $Z_0=R\times\{0\}$. Since $Z_0$
is $\pi_1$-injective in $Y$ by Property (\ref{his valet de sham}), it
follows that $c\times\{0\}$ is $\pi_1$-injective in $Y$, and hence
that $A$ is $\pi_1$-injective in $Y$.

To prove $\pi_1$-injectivity of $A$ in $\overline{X''-Y}$, we observe
that by construction we have
$\overline{X''-Y}=\overline{\Sigma -R}\times I
\subset(\Sigma \times I)\cup K=X''$. Since $A=c\times I$, it
suffices to prove that the boundary component $c$ of $R$ is
$\pi_1$-injective in $\overline{\Sigma -R}$, or equivalently
that no component of $\overline{\Sigma -R}$ is a disk. But
this follows immediately from
\ref{when an innocent heart}. Thus Property (\ref{for greater precision})
of $Y$ is established.
\EndProof

\Remark\label{so's your aunt tilly}
It follows from Property (\ref{and i who was}) and (\ref{his valet
de sham}) of $Y$, as stated in the conclusion of Lemma \ref{as steward}, that the
inclusion homomorphism $H_1(Z_0;\ZZ_2)\to H_1(Y;\ZZ_2)$ is an
isomorphism.
\EndRemark

\Lemma\label{one last kiss}
Suppose that $X$ is a trimonic  manifold relative to $S$. Then
$\partial X$ has exactly one
component $T\ne S$, whose genus is equal to that of $S$, and $T$ is $\pi_1$-injective in $X$.
\EndLemma

\Proof
Let us fix a compact, connected $3$-dimensional submanifold $Y$ of $X$
having Properties (\ref{for greater precision})--(\ref{oh baby})
stated in the conclusion of Lemma \ref{as steward}. Define $Z_0$ and
$Z_1$ as in the statement of Property (\ref{without the elision}) of
that lemma.  According to Property (\ref{without the elision}), we
have $Z_0\subset S$, and $Z_1$ is contained in a component $T\ne S$ of
$\partial X$.  If $Q$ is a component of $\overline{X-Y}$, then by Property
(\ref{oh baby}) of $Y$, $Q$ is a trivial $I$-bundle.  By Property 
(\ref{for greater precision}), each component of the frontier of
$Q$ is an essential annulus with one boundary component in
$Z_0$ and one in $Z_1$.  Thus
the components of $\partial_hQ$ may be labeled $F_0$ and $F_1$ in such
a way that
 $F_0$ meets $Z_0 \subset S$ and $F_1$ meets
$Z_1 \subset T$.
Note that $F_0$ and $F_1$
are homeomorphic.  Since this holds for each component of $\overline{X-Y}$,
and since $Z_0$ and $Z_1$ are homeomorphic, it follows that
$S$ and $T$ have the same Euler characteristic, and furthermore that
they are the only components of $\partial X$.

It remains to show that $T$ is $\pi_1$-injective in $X$. According to
Proposition \ref{easy come}, it suffices to show that the inclusion
homomorphism $H_1(S;\ZZ_2)\to H_1(M;\ZZ_2)$ is surjective. For this
purpose we set $\calq=\overline{X-Y}$ and $\partial_0 \calq=\calq \cap
S$, we let $F$ denote the frontier of $Y$ in $X$, and we consider the
commutative diagram
$$\xymatrix{
H_1(F\cap S) \ar@{->}[r] \ar@{->}^{\alpha_1}[d] &
 H_1(Z_0)\oplus H_1(\partial_0 \calq)\ar@{->}[r]\ar@{->}^{\beta_1\oplus\gamma_1}[d]&
H_1(S) \ar@{->}[r] \ar@{->}^{\delta}[d] &
H_0(F\cap S)\ar@{->}[r]\ar@{->}^{\alpha_0}[d]&
H_0(Z_0)\oplus H_0(\partial_0 \calq)  \ar@{->}^{\beta_0\oplus\gamma_0}[d] \\
H_1(F) \ar@{->}[r] &
H_1(Y)\oplus H_1(\calq)\ar@{->}[r]&
H_1(M) \ar@{->}[r]&
 H_0(F)\ar@{->}[r]&
H_0(Y)\oplus H_0(\calq),
}
$$
where all homology groups are defined with $\ZZ_2$-coefficients, the
rows are segments of Mayer-Vietoris exact sequences, and the
homomorphisms $\alpha_i$,
$\beta_i$, $\gamma_i$ and $\delta$  are induced by inclusion.

Since $Z_0\subset S$ and $Z_1\subset T$, each component of $F$ is an
annulus with exactly one boundary curve in $S$. Hence the maps
$\alpha_0$ and $\alpha_1$ are isomorphisms.

If $Q$ is any component of $\calq$, then  Property (\ref{oh baby})
of $Y$, as stated in Lemma
\ref{as steward}, implies that $Q$ may be given the structure of a
trivial $I$-bundle over a $2$-manifold in
such a way that  $Q\cap \partial X$ is the horizontal boundary of
$Q$. Since $Q$ must contain at least one component of $Y$, exactly one
component of the horizontal boundary of $Q$ lies in $S$. It follows that
$\gamma_0$ and $\gamma_1$ are isomorphisms.
The map $\beta_0$ is an isomorphism because $Z_0$ and $Y$ are both
connected, while $\beta_1$ is an isomorphism by Remark \ref{so's your
aunt tilly}.

Since the $\alpha_i$, $\beta_i$ and $\gamma_i$ are isomorphisms, it
follows from the Five Lemma that $\delta$ is an isomorphism. In
particular it is surjective, as required.
\EndProof

\Lemma\label{it can't happen here} Suppose that $X$ is a non-degenerate
trimonic  manifold relative to a component $S$ of $\partial X$. 
Then there is no normal book of
$I$-bundles $\calw$ with $|\calw|=X$.
\EndLemma

\Proof
Let us fix a compact, connected $3$-dimensional submanifold $Y$ of $X$
having Properties (\ref{for greater precision})--(\ref{oh baby})
of the conclusion of Lemma \ref{as steward}. Define $Z_0$ as in the
statement of Property (\ref{without the elision}) there.

Suppose that $X=|\calw|$ for some normal book of $I$-bundles
$\calw$. Then Properties (\ref{for greater precision}) and
(\ref{without the elision}) of $Y$, as stated in \ref{as steward}, give hypotheses
(\ref{i'm now employed}) and (\ref{the devil may take him}) of Lemma
\ref{as a new-born lamb}.
 According to Remark \ref{so's your aunt tilly}, the inclusion
homomorphism $H_1(Z_0;\ZZ_2)\to H_1(Y;\ZZ_2)$ is an isomorphism. This
is Hypothesis (\ref{i'll never forsake him}) of Lemma \ref{as a
new-born lamb}.

Let us fix a base point $\star\in Z_0$.
According to the properties (\ref{without the elision})
(\ref{and i who was}) and (\ref{his valet de sham}) of $Y$ stated in the
conclusion of Lemma \ref{as steward}, 
$Z_0$ is a planar surface with three boundary curves, and 
the group $\pi_1(Y)$ is free of rank $2$, and
 there exists a pair of
generators $(x,y)$ of $\pi_1(Y,\star )$, such that either the pair
$(x,y x^{-1}y^2)$ or the pair $(x,y^{-1} x^{-1}y^2)$ is the image of
some geometric basis of $\pi_1(Z_0,\star )$ under the inclusion
homomorphism $\pi_1(Z_0,\star )\to\pi_1(Y,\star )$.
 Since neither of the pairs 
$(x,y x^{-1}y^2)$ or $(x,y^{-1} x^{-1}y^2)$
generates
the free group on $x$ and $y$,
the inclusion homomorphism
  $\pi_1(Z_0,\star )\to\pi_1(Y,\star )$ is not surjective.
 Hence Lemma \ref{as a new-born lamb}
implies that there is a solid torus
$L\subset Y$ such that $A\doteq L\cap Z_0$ is an annulus which is
homotopically non-trivial in $|\calw|$, and the
inclusion homomorphism $H_1(A;\ZZ)\to H_1(L,\ZZ)$ is not
surjective. 

Let $c$ denote a core curve of the annulus $A$. Since $c$ is in
particular homotopically non-trivial in $Z_0$, which is a planar
surface with three boundary curves, $c$ is parallel to one of the
boundary curves of $Z_0$. In view of the definition of a geometric
basis, it follows that the conjugacy class in $\pi_1(Y,\star)$ defined
by a suitably chosen orientation of $c$ is represented by one of the
elements $x$, $y x^{-1}y^2$, $x^{-1}y x^{-1}y^2$, $y^{-1} x^{-1}y^2$
or $x^{-1}y^{-1} x^{-1}y^2$.  On the other hand, since the inclusion
homomorphism $H_1(A;\ZZ)\to H_1(L,\ZZ)$ is not surjective, a
representative of a conjugacy class in $\pi_1(Y,\star)$ defined by $c$
must be an $n$-th power in $\pi_1(Y,\star)$ for some $n\ne\pm1$. But
none of the elements $x$, $y x^{-1}y^2$, $x^{-1}y x^{-1}y^2$, $y^{-1}
x^{-1}y^2$ or $x^{-1}y^{-1} x^{-1}y^2$ is a proper power in the free
group on $x$ and $y$.   This contradiction completes the proof.
\EndProof


\section{With a $(1,1,1)$ hexagon.}  \label{sec:111}

In this section and the next we will be working with
hyperbolic $3$-manifolds with totally geodesic boundary. As is
customary in low-dimensional topology, we shall implicitly carry over
the PL results proved in Sections \ref{sec:trimonic} and
\ref{sec:boibs} to the smooth category. For example, to say that a
smooth manifold $X$ is a trimonic  manifold
relative to a component $S$ of $\partial X$ means that the pair
$(X,S)$ is topologically homeomorphic to a PL pair $(X'S')$ such that
$X'$ is a trimonic  manifold
relative to  $S'$. Likewise, to say that a smooth manifold $X$ has the form
$|\calw|$ for some normal book of $I$-bundles $\calw$ means that $X$
is topologically homeomorphic to a PL manifold which  has the form
$|\calw|$ for some normal book of $I$-bundles $\calw$.

Here we will describe some topological consequences of the presence
of a $(1,1,1)$ hexagon in $\widetilde{N}$, when the shortest return
path of $N$ is not too long.  The main result of the section,
Proposition \ref{111}, asserts the existence of a trimonic 
manifold in $N$ under these circumstances.  Below we prove a series of
separate lemmas concerning the geometry of $(1,1,1)$ hexagons, from
which the proposition follows quickly.  The first follows from
\cite[Lemma 3.2]{KM}, but for self-containedness we prove it here.

\begin{lemma}  \label{embeddedl1}  
Let $N$ be a compact hyperbolic $3$-manifold with totally geodesic
boundary, and $\lambda, \lambda' \subset \widetilde{N}$ short cuts
with length $\ell_1$.  Then $\lambda =\lambda'$ or $\lambda \cap
\lambda' = \emptyset$.  \end{lemma}

\begin{proof}  
Suppose $\lambda$ intersects $\lambda'$ at a single point $y \in
\lambda$.  Let $\Pi$ be the component of $\partial \widetilde{N}$
containing the endpoint $x$ of $\lambda$ closest to $y$, and let
$\Pi'$ be the component containing the endpoint $x'$ of $\lambda'$
closest to $y$.  The subarcs $[x,y]$ and $[x',y]$ of $\lambda$ and
$\lambda'$, respectively, each have length at most $\ell_1/2$.  They
meet at $y$ at an angle properly less than $\pi$, since $\lambda \neq
\lambda'$ and each is geodesic.  But then $\Pi$ and $\Pi'$ are
components of $\partial \widetilde{N}$ at distance less than $\ell_1$,
so they are equal.  But since $\lambda$ and $\lambda'$ are geodesic
arcs, each perpendicular to $\Pi$, in $\widetilde{N}$ they are
disjoint or equal, a contradiction.  \end{proof}

\Remark \label{reallyembeddedl1} 
Suppose $N$ is a compact hyperbolic $3$-manifold with totally geodesic
boundary and $g : \widetilde{N} \rightarrow \widetilde{N}$ is a
covering transformation.  If $\lambda$ is a short cut with length
$\ell_1$, then according to Lemma \ref{embeddedl1}, either $g(\lambda)
= \lambda$ or $g(\lambda) \cap \lambda = \emptyset$.  But in the first
case, $g$ would fix a point in $\lambda$, a contradiction.  It follows
that every short cut of length $\ell_1$ is embedded in $N$ by the
universal cover.  \EndRemark

\begin{lemma} 
\label{uniquel1} 
Let $N$ be a compact hyperbolic $3$-manifold with $\partial N$
connected, totally geodesic, and of genus $2$, and suppose $\ell_1$
satisfies the bound below.
$$\cosh \ell_1 < \frac{\cos(2\pi/9)}{2\cos(2\pi/9)-1} = 1.4396...$$
Then $\ell_2 > \ell_1$; ie, the shortest return path in $N$ is
unique. \end{lemma}

\begin{proof}  
Suppose $\ell_2 = \ell_1$.  Applying the right-angled hexagon rule as
in the proof of Lemma \ref{d11vsl1}, we find that each of $d_{11}$,
$d_{12}$, and $d_{22}$ is at least $2R$, for $R$ the function of
$\ell_1$ defined there.  Then a disk of radius $R$ is embedded around
each of the feet of $\lambda_1$ and $\lambda_2$, so that none of these
disks overlap.  Bor\"oczky's bound on the radius of four disks of
equal area embedded without overlapping on a surface of genus 2 is the
quantity $R''$ defined in Lemma \ref{KLM}.  Setting $R = R''$ and
solving for $\ell_1$ yields the quantity of the bound above.
\end{proof}

\begin{lemma}  
\label{noboundarycross}
Let $N$ be a compact hyperbolic $3$-manifold with totally geodesic
boundary.  Suppose $C$ and $C'$ are distinct $(1,1,1)$ hexagons in
$\widetilde{N}$ with exterior edges on the same component of $\partial
\widetilde{N}$.  Then $C \cap C'$ is either empty or a single interior
edge.  \end{lemma}

\begin{proof}  
Let $e$ and $e'$ be exterior edges of $C$ and $C'$ on the same
component of $\partial \widetilde N$.  The endpoints of $e$ and $e'$
are feet of lifts of the shortest return path; any such pair has
distance at least $d_{11}$.  Since $C$ and $C'$ are distinct, so are
$e$ and $e'$; if they share an endpoint then $C \cap C'$ consists of
an interior edge.

Otherwise, $e$ is a geodesic arc of length $d_{11}$ connecting its
endpoints $a$ and $b$ and $e'$ an arc of the same length connecting
its endpoints $a'$ and $b'$, with $d(a,a')$, $d(a,b')$, $d(b,a')$, and
$d(b,b')$ all at least $d_{11}$.  Some hyperbolic trigonometry shows
that any point at distance at least $d_{11}$ from $a$ and $b$
satisfies $\cosh \ell \geq \cosh d_{11}/\cosh (d_{11}/2),$ where
$\ell$ is its distance from $e$.  Twice this distance is larger than
$d_{11}$; thus $e'$ does not cross $e$, and so $C \cap C' =
\emptyset$.  \end{proof}

\Remark \label{embeddedboundary} 
Suppose $N$ is a compact hyperbolic $3$-manifold with totally geodesic
boundary and $g: \widetilde{N} \rightarrow \widetilde{N}$ is a
covering transformation.  If $C$ is a $(1,1,1)$ hexagon, and an
external edge of $g(C)$ intersects an external edge of $C$, then by
Lemma \ref{noboundarycross}, either $g(C) = C$ or $g(C) \cap C$ is a
single internal edge.  In the former case, $g$ fixes a point in $C$, a
contradiction.  It follows that the union of the interiors of external
edges of $C$ projects homeomorphically to $N$.  \EndRemark

\begin{lemma}\label{l1avoids111}  
Let $N$ be a compact hyperbolic $3$-manifold with totally geodesic
boundary, and suppose $\cosh \ell_1 \leq 1.215$.  If $C$ is a
$(1,1,1)$ hexagon in $\widetilde{N}$ and $\lambda$ is a short cut of
length $\ell_1$, then $\lambda$ is an internal edge of $C$ or $\lambda
\cap C = \emptyset$.  \end{lemma}

\begin{proof}  
Suppose $\lambda$ and $C$ are as in the lemma, and $\lambda \cap C$ is
nonempty.  Let $\Pi_1$ and $\Pi_2$ be the components of $\partial
\widetilde{N}$ containing the endpoints of $\lambda$, and let $\Pi$ be
the geodesic plane containing $C$.  Suppose first that $\lambda
\subset \Pi$.  If $\lambda$ is not contained in $C$, let $x$ be an
endpoint of $\lambda \cap C$ contained in the interior of $\lambda$.
Then $x$ is a point of intersection between $\lambda$ and an internal
edge of $C$, since the external edges of $C$ are contained in
$\partial \widetilde{N}$ and $\lambda$ is properly embedded.  But each
internal edge of $C$ is a short cut with length $\ell_1$, so this
contradicts Lemma \ref{embeddedl1}.  It follows that if $\lambda
\subset \Pi$, $\lambda \subset C$.  But $C$ intersects only three
components of $\partial \widetilde{N}$ --- the components joined by
its internal edges.  Hence since $C$ intersects $\Pi_1$ and $\Pi_2$,
$\lambda$ is an internal edge of $C$.

Now suppose $\lambda$ is not contained in $\Pi$.  Then $\lambda$
intersects $C$ transversely in a single point $x$.  There is a
component $\Pi'$ of $\partial \widetilde{N}$ such that $\Pi' \cap \Pi$
contains an external edge of $C$ and $x$ is distance at most $A$ from
$\Pi'$.  Then the distance from $\Pi'$ to each of $\Pi_1$ and $\Pi_2$
is less than $A + \ell_1$.  By Lemma \ref{l2twicel1}, $\ell_2$ is at
least twice $A$ and $\ell_1$; hence the distance from $\Pi'$ to each
of $\Pi_1$ and $\Pi_2$ is $\ell_1$, since it is less than $\ell_2$.
By Lemma \ref{geodhex}, there is a $(1,1,1)$ hexagon $C'$ with
$\lambda$ as an internal edge and external edges in $\Pi'$, $\Pi_1$,
and $\Pi_2$.  But then $C' \cap C$ contains $\lambda \cap C$,
contradicting Lemma \ref{noboundarycross}.
\end{proof}

\begin{lemma}  
Let $N$ be a compact hyperbolic $3$-manifold with totally geodesic
boundary, and suppose $\cosh \ell_1 \leq 1.215$.  If $C$ and $C'$ are
distinct $(1,1,1)$ hexagons in $\widetilde{N}$, then $C \cap C'$ is
empty or a single internal edge of each.  \label{Cembedded}
\end{lemma}

\begin{proof}  
Suppose $C$ and $C'$ are distinct $(1,1,1)$ hexagons in
$\widetilde{N}$, and $C \cap C' \neq \emptyset$.  By Lemma
\ref{noboundarycross}, the lemma holds if there is a component of
$\partial \widetilde{N}$ containing external edges of both $C$ and
$C'$, thus we may assume that this is not the case.  It follows that
no external edge of $C$ contains a point of $C \cap C'$ and
vice--versa, since by Lemma \ref{rtanghex}, $C \cap \partial
\widetilde{N}$ is precisely the union of its external edges.  Let
$\Pi$ be the geodesic hyperplane containing $C$.  If $C' \subset \Pi$,
then $C \cap C'$ is a two-dimensional subpolyhedron of $C$, and each
vertex of $C \cap C'$ is an intersection point between internal edges,
by the above.  But these are all short cuts of length $\ell_1$,
contradicting Lemma \ref{embeddedl1}.

If $C'$ is not contained in $\Pi$, it intersects $C$ transversely in a
geodesic arc, whose endpoints are points of intersection of internal
edges of one with the other.  But such intersections violate Lemma
\ref{l1avoids111}, since the internal edges of each are short cuts of
length $\ell_1$, a contradiction.  \end{proof}

\begin{proposition}\label{111}  
Let $N$ be a compact, orientable 
hyperbolic $3$-manifold with $\partial N$
connected, totally geodesic, and of genus $2$.  Suppose that $\cosh
\ell_1 \leq 1.215$, and there is a $(1,1,1)$ hexagon in
$\widetilde{N}$.  There is a submanifold $X \subset N$ 
with $\partial N\subset X$, such that $X$ is a 
trimonic
 manifold relative to $\partial N$.   
\end{proposition}

\begin{proof}  
Let $N$ be as in the hypotheses, and fix a $(1,1,1)$ hexagon $C
\subset \widetilde{N}$.  Let $f : C \rightarrow N$ be the restriction
of the universal covering, and let $X$ be a regular neighborhood of
$\partial N \cup f(C)$.  Since $\cosh \ell_1 \leq 1.215$, Lemma
\ref{uniquel1} implies that $N$ has a unique shortest return path
$\alpha$, hence each internal edge of $C$ projects to $\alpha$ by $f$.
The preimage of $\alpha$ is a union of short cuts with length
$\ell_1$, hence Lemma \ref{l1avoids111} implies Property (1) of
Definition \ref{i'm now sir murgatroyd}.  Remark
\ref{reallyembeddedl1} now immediately implies Property (2) of the
definition.  Property (3) follows from Remark \ref{embeddedboundary}
and Lemma \ref{Cembedded}.  Properties (4) and (5) follow from Lemma
\ref{rtanghex}, and (6) holds by construction.  \end{proof}

\begin{proposition} \label{Cboundary} 
Let $N$ be a compact,  
orientable hyperbolic $3$-manifold with $\partial N$
connected, totally geodesic, and of genus $2$, such that $\cosh \ell_1
\leq 1.215$ and there is a $(1,1,1)$ hexagon in $\widetilde{N}$.  The
trimonic  submanifold $X \subset N$ supplied by Proposition
\ref{111} is non-degenerate.  \end{proposition}


\begin{proof}  
Let $N$ satisfy the hypotheses, and as in the proof of Proposition
\ref{111} let $C \subset \widetilde{N}$ be a $(1,1,1)$ hexagon, $f : C
\rightarrow N$ the restriction of the universal cover, and $\alpha
\subset N$ the shortest return path.  Below we will borrow wholesale
the constructions and notation from the proof of Lemma \ref{the
magic}, with $C$ here in the role of $D$ there and $\partial N$ in the
role of $S$.

Recall that the internal edges $a_i$ and external edges $b_i$ of $C$,
$i \in \{0,1,2\}$, are enumerated so that $a_i$ shares a vertex with
$b_i$ and $b_{i-1}$ for each $i$, and $\partial C$ is oriented so that
$\phi_0 \doteq f|a_0$ induces the same orientation on $\alpha$ as
$\phi_1 \doteq f|a_1$.  Then $\alpha$ is given the orientation induced
by $\phi_0$ and $\phi_1$, with initial and terminal vertices $V$ and
$W$, respectively.  For each $i \in \{0,1,2\}$, the edge $\beta_i =
f(b_i)$ of $\Gamma = \Gamma_f$ 
is given the induced orientation from $b_i$.

Cases (I) and (II) in the proof of Lemma \ref{the magic} are
distinguished according to whether $\phi_2 \doteq f|a_2$ is
orientation preserving or reversing, respectively.  In Case (I),
$\Gamma = \overline{\beta_0 \cup \beta_1 \cup \beta_2}$ is a theta
graph, and in Case (II) it is an eyeglass.  (Recall that $\beta_i =
f(\mathrm{int}\ b_i)$, $i \in \{0,1,2\}$.)

For each $i \in \{0,1,2\}$, let $\Pi_i$ be the component of $\partial \widetilde{N}$ containing $b_i$.  
Suppose that $U$ is a component of $\partial N - \Gamma$ which is homeomorphic to an open disk, and let $U_0 \subset \Pi_0$ be a component of the preimage of $U$ under the universal covering map.  Then $U_0$ is projected homeomorphically to $U$, and $\overline{U}_0$ is a compact polygon in $\Pi_0$ with edges projecting to the $\beta_i$, hence covering translates of the $b_i$.  Since the $b_i$ are geodesic arcs, $\overline{U}_0$ has at least three edges.  Since each vertex $v$ of $\overline{U}_0$ projects to $V$ or $W$, and $U_0$ projects homeomorphically, $\overline{U}_0$ has at most 6 vertices.  Hence $\overline{U}_0$ is an $n$--gon for $n$ between $3$ and $6$.

Suppose edges $b$ and $b'$ incident to a single vertex $v$ of
$\overline{U}_0$ were identified in $\Gamma$.  The covering
transformation $f$ taking $b$ to $b'$ is orientation preserving on
$\widetilde{N}$, so it preserves the boundary orientation on $\Pi_0$.
Give $U_0$ this orientation, and orient $b$ and $b'$ as arcs in the
boundary of $U_0$.  Since $f(U_0)$ does not intersect $U_0$,
$f(\overline{U}_0) \cap \overline{U}_0 = b' = f(b)$.  Since $f(b)$ has
the boundary orientation from $f(U_0)$, its orientation is opposite
that of $b'$.  But then since $v$ is the initial vertex of (say) $b$
and the terminal vertex of $b'$, $f(v) = v$, a contradiction.  Hence:

\Claim  \label{Unidentified edges} If $v$ is a vertex of $\overline{U}_0$, 
the edges incident to $v$ project to distinct edges of $\Gamma$.  \EndClaim

Now suppose $\overline{U}_0$ is a triangle.  Then two of its vertices
are identified in $N$, since $\Gamma$ has only two vertices.  The edge
joining these vertices projects to an edge joining $V$ to $V$ or $W$
to $W$, so in this case $\Gamma$ is an eyeglass graph.  On the other
hand, the final vertex of $\overline{U}_0$ is not identified with the
other two by \ref{Unidentified edges}, since an eyeglass graph has
only one edge joining each vertex to itself.  But then the two edges
emanating from the final vertex yield distinct edges of $\Gamma$,
joining $V$ to $W$, which does not occur in an eyeglass graph.  This
is a contradiction.

When $\overline{U}_0$ is a pentagon, three of its vertices are
identified to $V$ (say) in $\Gamma$.  Thus two of these are adjacent
in $\overline{U}_0$.  The edge joining the adjacent vertices of
$\overline{U}_0$ identified to $V$ joins it to itself in $\Gamma$,
hence $\Gamma$ is an eyeglass graph.  The third vertex identified to
$V$ is not adjacent to either of the others, by \ref{Unidentified
edges}, since $\Gamma$ has only one edge joining $V$ to itself.  Then
the edges adjacent to this vertex project to distinct $\beta_i$
joining $V$ to $W$, a contradiction.

To rule out the possibility that $\overline{U}_0$ is a quadrilateral
or hexagon requires counting angles.  Recall that $S = \partial N$ is
oriented so that $\calt_V$ is a positive oriented triod (see
\ref{troglodyte}).  Here $\calt_V = (\sigma_0,\sigma_1,\sigma_2)$ in
Case (I) and $\calt_V = (\sigma_0,\sigma_1,\rho_2)$ in Case (II),
where for each $i \in \{0,1,2\}$, $\rho_i$ is the closure of
an initial segment of
$\beta_i$ and $\sigma_i$ is the closure of
a terminal segment of $\beta_{i-1}$.
Define $\theta_1$ to be the angle measure from $\sigma_0$ to
$\sigma_1$ at $V$, in the direction prescribed by the orientation on
$\partial N$.  In Case (I), we take $\theta_2$ to be the angle from
$\sigma_0$ to $\sigma_2$ in the orientation direction, and in Case
(II) we let $\theta_2$ be the angle from $\sigma_0$ to $\rho_2$.  Then
$0 < \theta_1 < \theta_2 < 2\pi$.

Recall that $C$ is a totally geodesic hexagon in $\widetilde{N}$ by
Lemma \ref{geodhex}, and the covering projection $f$ immerses $C$ 
in $N$ isometrically.  
Then appealing to Figure \ref{hexnT}, we note that $\theta_1$ is the
angle from $\rho_1$ to $\rho_0$ at $W$, measured in the orientation
direction.  This is because the homeomorphism $\psi:\delta_V
\rightarrow \delta_W$ defined in the proof of Lemma \ref{the magic} is
orientation reversing.  Similarly, in Case (I) $\theta_2$ is the the
angle from $\rho_2$ to $\rho_0$ at $W$ in the orientation direction,
and in Case (II) it is the angle from $\sigma_2$ to $\rho_0$.

If $\overline{U}_0$ is a quadrilateral, then two of its edges are sent by $f$ 
to the same edge of $\Gamma$.  These are opposite, by \ref{Unidentified
edges}; hence the image in $\Gamma$ of each of the remaining edges 
joins a vertex to itself.  Then $\Gamma$ is an eyeglass graph, the pair of 
opposite edges identified by $f$ project to $\beta_0$, and the other
two edges project to $\beta_1$ and $\beta_2$.  Abusing notation slightly,
we label the edges projecting to $\beta_0$ by $b_0$ and $b_0'$,
and similarly label the edge projecting to $\beta_i$ by $b_i$, $i = 1,2$.
(These are covering translates of the corresponding edges of $C$.)  
We give each edge of $\overline{U}_0$ an orientation matching that
of its correspondent in $\Gamma$.

Orient $U_0$ so that $f|U_0$ preserves orientation.  We may assume, 
by switching the labels of $b_0$ and $b_0'$ and/or replacing $C$ 
with a covering translate if necessary, that the prescribed orientation
on  $b_0 \subset C$ matches the boundary orientation which it inherits 
from $\overline{U}_0$.  The terminal endpoint of $b_0$ is sent
to $V$, and a terminal segment is sent to the interior of $\sigma_1$
by definition.  Since the orientation on $\partial N$
is chosen so that $(\sigma_0,\sigma_1,\rho_2)$ is a positively
oriented triod, the orientation $\sigma_1$ inherits from $\beta_0$
matches the boundary orientation from the component of $\delta_V -
(\sigma_0 \cup \sigma_1 \cup \sigma_2)$ bounded by it and $\rho_2$.
It follows that the terminal endpoint of $b_0$ is the initial endpoint
of $b_2$, and the dihedral angle of $\overline{U}_0$ at this vertex is
$\theta_2 - \theta_1$.

Since the orientation on $b_0'$ is opposite that induced by
$\overline{U}_0$, the dihedral angle of $\overline{U}_0$ at the common
terminal endpoint of $b_0'$ and $b_2$ is $\theta_1$, the dihedral
angle in $\delta_V$ between $\sigma_0$ and $\sigma_1$.  Arguing as
above, we find that the dihedral angle of $\overline{U}_0$ at the
initial endpoint of $b_0$, which is the terminal endpoint of $b_1$, is
$2\pi - \theta_2$.  (Recall that the homeomorphism $\psi: \delta_V
\rightarrow \delta_W$ visible in Figure \ref{hexnT} as projection
upward is orientation reversing.)  The dihedral angle at the final
vertex is then $\theta_1$.  Thus the sum of the dihedral angles is
$2\pi + \theta_1$, contradicting the well known fact that the dihedral
angle sum of a hyperbolic quadrilateral is less than $2\pi$.

Now suppose $\overline{U}_0$ is a hexagon.  Then the projection of
$\overline{U}_0$ contains each of $\delta_V$ and $\delta_W$, since
$\overline{U}_0$ has $6$ vertices and each of $V$ and $W$ has valence
$3$.  Thus the sum of the dihedral angles around vertices of
$\overline{U}_0$ is $4\pi$.  But a hyperbolic hexagon has dihedral
angle sum less than $4\pi$, a contradiction.  It follows that no
component of $\partial N - \Gamma$ is homeomorphic to an open disk.
Since $f : C \rightarrow N$ is a defining hexagon (in the sense of
Definition \ref{non-degenerate}) for the submanifold $X$ defined in 
Proposition \ref{111}, and since $\Gamma=\Gamma_f$, the
trimonic manifold $X$ is non-degenerate.  \EndProof


\section{Putting it all together}  \label{sec:closed}

In this section we prove the theorems stated in the introduction.  
Here we make much use of terminology and results from \cite{CDS}.  Of
particular importance is the term ``$(g,h)$-small'', see Definition
1.2 there.

\begin{lemma}  \label{I-bundle genus two boundary}  
Let $\calw$ be a normal book of $I$--bundles, 
set $W = |\calw|$,
and suppose that $\partial W$ is connected and has genus $2$.
Then $\mathrm{Hg}(W) = 3$.  \end{lemma}

\begin{proof}  
Since $\chibar(W) = \frac{1}{2}\chibar(\partial W) = 1$, there 
is a unique page $P$ of $\calw$ which has negative Euler
characteristic; and furthermore, $\chibar(P) = 1$.   
Since $P \cap \partial W =
\partial_h P$ is a $\pi_1$-injective subsurface of 
$\partial W$ with Euler characteristic $-2$, its complement in
$\partial W$ is a disjoint union of annuli.  Thus if $B$ is a
component of $\overline{W-P}$, then $\partial B$ is a union
of annuli in $\overline{\partial W - P}$ and vertical annuli in the
frontier of $P$, and is therefore a torus.  Since the
frontier of $B$ in $W$ consists of essential annuli, $B$ is
$\pi_1$-injective in $W$, and since $W$ is simple, $B$ is
$(2,2)$-small.  It now follows from \cite[Proposition
2.3]{CDS} that each component of $\overline{W-P}$ is a solid torus.

Let $C$ be a closed disk contained in the interior of $T$.  If $p: P
\rightarrow T$ is the bundle projection, $\calh = p^{-1}(C)$ is a
1-handle in $P$ joining $\partial_h P \subset \partial W$ to itself.
Let $\delta_0$ be an arc embedded in $\overline{T-C}$, so that
$\partial \delta_0 = \delta_0 \cap \partial (\overline{T-C}) \subset
\partial C$ and no arc of $\partial C$ bounds a disk embedded in
$\overline{T-C}$ together with $\delta_0$.  The existence of such an
arc can be established using the fact that $\chi(\overline{T-C}) = -2$
and standard Morse theory arguments.  If $D_0$ is a regular
neighborhood of $\delta_0$ in $\overline{T-C}$, then
$\overline{T-C-D_0}$ is a possibly disconnected surface with Euler
characteristic $-1$ and no component which is a disk.  Let $T_0$ be
the component with $\chi(T_0) = -1$, and let $\beta$ be a component of
$\partial T_0$ containing an arc of the frontier in $\overline{T-C}$
of $D_0$.  There is an arc $\delta_1$ embedded in $T_0$, so that
$\partial \delta_1 = \delta_1 \cap \partial T_0 \subset \beta$ and no
arc of $\beta$ bounds a disk in $T_0$ together with $\delta_1$.  This
follows as above, and we may further assume, after sliding $\partial
\delta_1$ along $\beta$ if necessary, that $\partial \delta_1$ does
not intersect the frontier of $D_0$.  Let $D_1$ be a regular
neighborhood in $T_0$ of $\delta_1$ which does not intersect $D_0$,
and let $D = D_0 \sqcup D_1 \subset \overline{T-C}$.

By construction, $\overline{T-C-D}$ is a disjoint union of surfaces
with Euler characteristic $0$; that is, annuli and M\"obius bands.
Since $T$ is connected, each component of $\overline{T-C-D}$ has at
least one component of its boundary containing arcs of the frontier of
$D$.  Thus if $\alpha$ is a component of $\partial T$, the component
$T'$ of $\overline{T-C-D}$ containing $\alpha$ is an annulus, and
$\alpha$ is the unique component of $\partial T'$ contained in
$\partial T$.  For $i = 0$ or $1$, let $\cald_i = p^{-1}(D_i) \subset
P$, and let $\cald = \cald_0 \sqcup \cald_1$.  Each of $\cald_0$ and
$\cald_1$ is an $I$-bundle over a disk, hence a ball, and each
component of $\overline{P-\calh - \cald}$ is an $I$-bundle over an
annulus or M\"obius band, hence is a solid torus.

Let $B$ be a component of $\overline{W - P}$. Then the component of
$\overline{W -\calh - \cald}$ containing $B$ is the union of $B$ with
a collection of solid torus components of $\overline{P-\calh -
\cald}$.  If $B_1$ is such a component, by the above $B_1$ is an
$I$-bundle over an annulus component of $\overline{T-C-D}$ with a
unique boundary component $\alpha \subset \partial T$.  Let $A =
p^{-1}(\alpha) = B \cap B_1$, a vertical annulus in the frontier of
$P$.  Since $A$ is a degree one annulus in $\partial B_1$ it follows
that $B \cup B_1$ is a solid torus.  It follows that each component of
$\overline{W - \calh - \cald}$ is a solid torus, so $\overline{W -
\calh}$ is obtained from a collection of solid tori by adding
$\cald_0$ and $\cald_1$.

Let $\caln_0$ be a regular neighborhood of $\partial W$ such that
$\caln_0 \cap P$ is a regular neighborhood of $\partial_h P$ in $T$
with horizontal frontier.  Then $\caln_0 \cap \calh$ is a disjoint
union of two solid cylinders.  Let $V_0 = \caln_0 \cup \calh$.  $V_0$
is a compression body in $W$ with frontier a surface $S$ of genus 3.
Our description above shows that $V_1 = \overline{W-V_0}$ is the union
of a collection of solid tori with $\cald_0 \cap V_1$ and $\cald_1
\cap V_1$.  Each of these has the structure of a 1-handle, since it is
a ball and its intersection with its complement in $V_1$ consists of
two disks.  Hence $V_1$ is a handlebody and $S$ is a Heegaard surface
for $W$.  \end{proof}

\begin{theorem} \label{hg4orvol7.32} \hgfourorvolseventhreetwo \end{theorem}

\begin{proof}  
Let $N$ satisfy the hypotheses of the theorem, and let $X \subset N$
be the codimension-0 submanifold supplied by Proposition \ref{111},
which is a trimonic  manifold with respect to $\partial N$,
nondegenerate by Proposition \ref{Cboundary}.  Let $T$ be the frontier
of $X$ in $N$.  By Lemma \ref{one last kiss}, $T$ is a surface of
genus $2$ which is $\pi_1$-injective in $X$.

Let $V = \overline{N - X}$.  Then $V$ is a compact,
connected, irreducible 3-dimensional submanifold of $N$ which is
$\pi_1$-injective, with $\partial V = T$.  Therefore $\chibar(V) = 1$.
Note that $N$ is $(2,2)$-small, since it admits a hyperbolic structure
with geodesic boundary.  Thus in the case where $V$ is
boundary-reducible, it is a handlebody by
\cite[Proposition 2.3]{CDS}.  By Lemma \ref{as steward}, $X$ has
Heegaard genus equal to $3$. Then in this case, a genus 3 Heegaard
surface for $X$ is a genus 3 Heegaard surface for $N$ (cf. \cite[Lemma
2.1]{CDS}), hence $\Hg(N) = 3$.

Now  consider the case in which $V = |\calw|$ for some
normal book of $I$--bundles $\calw$. 
By Lemma \ref{I-bundle genus two
boundary}, we have $\Hg(V) =3$.  Amalgamating the Heegaard splittings of $V$
and $X$, each of genus 3, across $T$ yields a Heegaard splitting of
$N$ with genus 4 (cf. \cite{Schul1}, Remark 2.7 and the definition
above it).

There remains the case in which $V$ is boundary-irreducible but is not
homeomorphic to $|\calw|$ for any book of $I$-bundles $\calw$. Since
$V$ is boundary-irreducible and $T$ is $\pi_1$-injective in $X$, the
surface $T$ is incompressible in $N$. Hence $V$ and $X$ are simple. By Lemma
\ref{it can't happen here}, $X$
is also not homeomorphic to $|\calw|$ for any book of $I$-bundles
$\calw$.  Hence by Proposition \ref{all god's
 chillun got guts}, if
$\Sigma_V$ and $\Sigma_X$ denote the characteristic submanifolds of
$V$ and $X$ relative to their boundaries, we have
$\chi(\overline{X-\Sigma_X})<0$ and
$\chi(\overline{V-\Sigma_V})<0$. According to \cite[Definition
1.1]{CDS}, $\kish(V)$ (or $\kish(X)$) is the union of all components
of $\overline{V-\Sigma_V}$ (or respectively
$\overline{X-\Sigma_X}$) having negative Euler characteristic. We
therefore have $\kish V\ne\emptyset$ and $\kish X\ne\emptyset$, so that
$\chibar(\overline{X-\Sigma_X})\geq 1$ and
$\chibar(\overline{V-\Sigma_V})\ge1$. Hence
$\chibar(\kish N \cut T) =\chibar(\kish(X))+\chibar(\kish(V))
\geq 2$, and by
\cite[Theorem 9.1]{ASTD}, the volume of $N$ is greater than $7.32$.
\end{proof}

\newtheorem*{vol6.89Thm}{Theorem \ref{vol6.89}}
\begin{vol6.89Thm} \volsixeightnine
\end{vol6.89Thm}

\begin{proof}  
Let $N$ satisfy the hypotheses of the theorem, and let $\ell_1$ be the
length of the shortest return path of $N$.  If $\cosh \ell_1 \geq
1.215$, then by Proposition \ref{l1cosh1.215}, $N$ has volume greater
than $6.89$.  If $\widetilde{N}$ contains no $(1,1,1)$ hexagon, then
by Proposition \ref{no(1,1,1)}, $\ell_1 \geq 1.215$, and Proposition
\ref{l1cosh1.215} again gives the desired volume bound.  We thus
suppose that $N$ has a $(1,1,1)$ hexagon and $\cosh \ell_1 < 1.215$.
But in this case Theorem \ref{hg4orvol7.32} gives a better volume
bound of $7.32$, since by hypothesis $\Hg(N) \geq 5$.  \end{proof}

\newtheorem*{genus2or3Thm}{Theorem \ref{genus2or3}}
\begin{genus2or3Thm}  \genustwoorthree
\end{genus2or3Thm}

\begin{proof}  
Let $M$ satisfy the hypotheses of Theorem \ref{genus2or3}, and note
that since $M$ is simple, it is $(2,2)$-small by definition.  Suppose
$M$ contains a connected closed incompressible surface of genus 2.  If
$M$ is $(3,2)$-small, the hypothesis on $\Hg(M)$ and
\cite[\CDSnonsep]{CDS} imply that for any such surface $S$,
$\chibar(\kish(M\cut S)) \geq 2$, satisfying the first conclusion of
Theorem \ref{genus2or3}.  Otherwise, \cite[\CDStrichotomy]{CDS}
provides a separating, connected, closed incompressible surface $S$ of
genus $2$ satisfying one of the conditions below.  
\begin{enumerate}
\item  At least one component of $M \cut S$ is acylindrical; or
\item  For each component $B$ of $M \cut S$ we have $\kish(B) \neq
\emptyset$.
\end{enumerate}
If $S$ satisfies condition $(2)$, then since each component $B$ of $M
\cut S$ has $\kish B \neq \emptyset$ we have $\chibar(\kish(M\cut S))
\geq 2$, which implies the first conclusion of Theorem
\ref{genus2or3}.  We address the other case below.

Now suppose that $M$ contains no connected closed incompressible
surface of genus 2 but contains a connected closed incompressible
surface of genus $3$.  If $M$ is $(5,3)$-small, the hypothesis on
$\Hg(M)$ and \cite[\CDSnonsep]{CDS} imply that for any such surface
$S$, $\chibar(\kish(M\cut S)) \geq 4$, satisfying the first conclusion
of Theorem \ref{genus2or3}.  The remaining possibilities are that $M$
contains a separating incompressible surface of genus $g$ and is
$(g,3)$-small, for $g = 3$ or $4$.  In either case, the hypothesis on
Heegaard genus ensures that \cite[\CDStrichotomy]{CDS} provides a
separating connected closed incompressible surface $S$ of genus $g$
satisfying condition $(1)$ or $(2)$ above.  As above, if $S$ satisfies
conclusion $(2)$, then the first conclusion of Theorem \ref{genus2or3}
follows.

In the remaining cases, we have a separating, connected, closed
incompressible surface $S \subset M$ of genus $2$, $3$, or $4$,
satisfying condition $(1)$ above, and we may assume that $S$ does not
satisfy condition $(2)$ there.  Let $N$ be an acylindrical component
of $M \cut S$ and $B$ the remaining component.  Since $N$ is
acylindrical, $N=\kish(N)$.  Therefore $\kish(B) = \emptyset$, since
otherwise $S$ would satisfy condition $(2)$ above.  Then $B = |\calb|$
for some book of $I$-bundles $\calb$ (cf. \cite[\S 5.1]{CDS}), and so
by \cite[Lemma 5.3]{CDS}, $B$ is ``shallow relative to $S$''
(\cite[Definition 4.3]{CDS}).  It now follows from \cite[Lemma
4.4]{CDS} that $\Hg(M) \leq 1 + \Hg(N)$, hence by hypothesis that
$\Hg(N) \geq 7$.  Thus in this case the second conclusion of Theorem
\ref{genus2or3} holds.  \end{proof}

\newtheorem*{closedvol6.89Thm}{Theorem \ref{closedvol6.89}}
\begin{closedvol6.89Thm} \closedvolsixeightnine
\end{closedvol6.89Thm}

\begin{proof}  
We apply Theorem \ref{genus2or3} to $M$, yielding a connected closed
surface $S$ of genus at most $4$ satisfying its conclusion.  If
$\chibar(\kish(M\cut S)) \geq 2$, then Theorem 9.1 of \cite{ASTD}
implies that the volume of $M$ is greater than $7.32$.  Thus we assume
$S$ is separating and $M \cut S$ has an acylindrical component $X$
with $\Hg(X) \geq 7$.  It is a standard result (cf. \cite[Proposition
6.3]{CDS}) that $X$ is homeomorphic to a hyperbolic $3$-manifold $N$
with totally geodesic boundary, and $\mathrm{vol}(N) =
\mathrm{geodvol}(X)$ (see Definition 6.2 of \cite{CDS}).

If $S$ has genus at least 3, then by Miyamoto's Theorem \cite[Theorem
5.4]{Miy}, $\mathrm{vol}(N) > 10.4$.  If $S$ has genus 2, then Theorem
\ref{vol6.89} implies that $N$ has volume greater than $6.89$.
Theorem \ref{closedvol6.89} now follows from \cite[Proposition
6.4]{CDS} (which is in turn derived from results in \cite{ASTD}).
\end{proof}

\newtheorem*{vol3.44Thm}{Theorem \ref{vol3.44}}
\begin{vol3.44Thm} \volthreefourfour
\end{vol3.44Thm}

\begin{proof}  
If $\pi_1 M$ is $4$--free, then Theorem 1.2 of \cite{CS_vol3.44}
implies that $M$ has volume greater than $3.44$.  Otherwise there is a
subgroup $G$ of $ \pi_1 M$ which has rank at most $4$ and
is not free.  The
homological hypotheses and Proposition 3.5 of \cite{CS_onecusp} ensure
that there is a twofold cover $\widetilde{M} \rightarrow M$, with $
\mathrm{dim}_{\mathbb{Z}_2} H_1(\widetilde{M}; \mathbb{Z}_2) \geq 8,$
such that $G < \pi_1 \widetilde{M}$.  Theorem 1.1 of \cite{CS_vol3.44}
implies that $\widetilde{M}$ contains an incompressible surface of
genus 2 or 3.  Since $\Hg(\widetilde{M})$ bounds above the dimension
of its $\mathbb{Z}_2$-homology, we have $\Hg(\widetilde{M}) \geq 8$.
Theorem \ref{closedvol6.89} now implies that $\widetilde{M}$ has
volume greater than $6.89$; hence that $M$ has volume greater than
$3.445$.  \end{proof}


\bibliographystyle{plain}
\bibliography{volumes}

\end{document}